\documentclass[11pt]{amsart}
\usepackage{amsfonts}
\usepackage{graphicx}
\usepackage{xcolor}
\usepackage{amsmath}
\usepackage{amssymb}
\usepackage{mathtools}
\usepackage[normalem]{ulem}
\newcommand{\st}[1]{\ifmmode\text{\sout{\ensuremath{#1}}}\else\sout{#1}\fi}

\usepackage{enumitem}
\setlist[enumerate]{topsep=2mm,parsep=0pt,partopsep=0mm,itemsep=1mm,leftmargin=15mm,labelsep=2mm}

\usepackage{xspace}

\usepackage[hidelinks]{hyperref}
\usepackage{geometry}

\theoremstyle{plain}
\newtheorem{theorem}{Theorem}[section]
\newtheorem{proposition}[theorem]{Proposition}

\newtheorem{corollary}[theorem]{Corollary}
\newtheorem{lemma}[theorem]{Lemma}
\newtheorem{result}{Theorem}
\theoremstyle{definition}
\newtheorem{remark}[theorem]{Remark}

\newtheorem{definition}[theorem]{Definition}
\newtheorem{example}[theorem]{Example}

\numberwithin{equation}{section}

\newcommand{\proofof}[1]{\fontseries{bx}\fontshape{it}\selectfont Proof of #1}
\newcommand{\proofofNP}[1]{\fontseries{bx}\fontshape{it}\selectfont Proof of #1\nopunct}

\let\Re\undefined

\DeclareMathOperator{\Re}{\mathtt{Re\hskip-.07ex}}

\let\ge=\geqslant
\let\le=\leqslant

\newcommand{\C}{{\mathbb{C}}}
\newcommand{\R}{{\mathbb{R}}}
\newcommand{\N}{{\mathbb{N}}}
\let\Complex=\C
\let\Real=\R
\let\Natural=\N

\newcommand{\UH}{{\mathbb{H}}} % the right half-plane
\newcommand{\U}{\mathfrak{U}\hskip.05em}
\newcommand{\UD}{{\mathbb{D}}}  %the unit disk

\newcommand{\di}{\mathrm{d}}  % for integration, e.g. \di x

\newcommand{\Hol}{{\sf Hol}}
\newcommand{\id}{{\sf id}}
\newcommand{\anglim}{\angle\lim}

\newcommand{\DCn}{$D\in\{\UD,\UH\}$\xspace}

\newcommand{\Lloc}{L^1_{\mathrm{loc}}}
\newcommand{\ind}{\mathbf{1}} % indicator function
\newcommand{\xx}{x} % for real variable of functions on the unit disk

\renewcommand{\Pr}{{\sf P}} % the set of probability measures on ...
\newcommand{\PGF}{\mathfrak{PGF}}  % the set of probability generating functions
\newcommand{\Pp}{\mathfrak{P}^+} % the set of strict probability generating functions
\newcommand{\BF}{\mathfrak{BF}} % the set of Bernstein functions without the identically zero function
\newcommand{\Gen}{\mathcal{G}} % the set of infinitesimal generators

\newcommand{\cadlag}{c\`adl\`ag\xspace}

\newcommand{\starM}{\circ} % composition of transition kernels
\newcommand{\LpT}[1]{\mathcal{L}[#1]} % Laplace transform

\newcommand{\StepN}[1]{\textsc{Step~#1}}
\newcommand{\Step}[1]{\medskip\noindent{\StepN{#1}.}}
\newcommand{\StepP}[1]{\medskip\noindent{\textsc{Proof of #1.}}}

\DeclareRobustCommand{\SkipTocEntry}[5]{}

\begin{document}

\title[Loewner Theory for Bernstein functions II]{Loewner Theory for Bernstein functions II: applications to inhomogeneous continuous-state branching processes}
\author[P.~Gumenyuk, T.~Hasebe, J.\,L. P\'erez]{Pavel Gumenyuk, Takahiro Hasebe, Jos\'e-Luis P\'erez}

\address{P.~Gumenyuk: Department of Mathematics, Politecnico di Milano, via E.~Bonardi~9, 20133 Milan, ITALY} \email{pavel.gumenyuk@polimi.it}

\address{T.~Hasebe: Department of Mathematics,
Hokkaido University,
North~10, West~8, Kita-ku,
Sapporo 060-0810, JAPAN} \email{thasebe@math.sci.hokudai.ac.jp}

\thanks{This work is supported by the Research Institute for Mathematical Sciences,
an International Joint Usage/Research Center located in Kyoto University, and by JSPS Grant-in-Aid for Young Scientists 19K14546, JSPS Scientific Research 18H01115 and JSPS Open Partnership Joint Research Projects grant no.~JPJSBP120209921.}

\address{J.\,L. P\'erez: Department of Probability and Statistics, Centro de Investigaci\'on en Matem\'aticas, A.C., Calle Jalisco s/n, C.P. 36240, Guanajuato, MEXICO} \email{jluis.garmendia@cimat.mx}

\begin{abstract}
This paper continues the research project launched in~[Constr.\ Approx.\ (2025) \href{https://doi.org/10.1007/s00365-023-09675-9}{DOI~10.1007/s00365-023-09675-9}]
 and aimed at studying time-inhomogeneous one-dimensional branching processes (mainly on a continuous but also on a discrete state space) with the help of recent achievements in Loewner Theory dealing with evolution families of holomorphic self-maps in simply connected domains of the complex plane. Under a suitable stochastic continuity condition, we show that the families of the Laplace exponents of branching processes on~$[0,\infty]$ can be characterized as topological (i.e.\ depending continuously on the time parameters) reverse evolution families whose elements are Bernstein functions.
For the case of a stronger regularity w.r.t. time,  we establish a Loewner\,--\,Kufarev type ODE for the Laplace exponents and characterize branching processes with finite mean in terms of the vector field driving this~ODE. Similar results are obtained for families of probability generating functions of branching processes on the discrete state space $\{0,1,2,\ldots\}\cup\{\infty\}$. In addition, we find a necessary and sufficient condition for ``spatial'' embeddability of such branching processes into branching processes on~$[0,\infty]$.
Finally, we give some probabilistic interpretations of the Denjoy\,--\,Wolff point at~$0$ and at~$\infty$.
\end{abstract}

\maketitle

\tableofcontents

\newpage

\section{Introduction}\label{sec1}
Branching processes are certain Markov processes
with various applications including biological population growth e.g.\ in epidemic models \cite{Bri10}. The case of discrete time and one-dimensional discrete state (i.e.\ the set of non-negative integers) was rather classical and it goes back to the work of Bienaym\'e, Galton and Watson in the 19th century; see \cite{AN72,Har63,Ken75}.
The continuous time setting~--- mostly in the \emph{time-homogeneous} case~--- also has been intensively studied in the literature;  we refer the reader to \cite{AN72,Har63} for the discrete space case and to~\cite{Kyp14} for the one-dimensional continuous state case. Time-homogeneous branching processes on multi-dimensional continuous state spaces have been also studied, see e.g. \cite{RyzhovSkorokhod, LiLi2024} and references therein.

In the continuous state setting, a crucial tool for analyzing branching processes is one-parameter (compositional) semigroups of Laplace exponents of transition kernels. The Laplace exponents in the one-dimensional case are Bernstein functions, which are known to appear as the Laplace exponents of subordinators too. Hence,  one-parameter semigroups of Bernstein functions seem to play an important role in this context.

Time-inhomogeneous branching processes have been also investigated; see e.g.\ \cite{Har63,MJ1958} and references therein, but this area seems to be considerably less explored compared with the time-homogeneous case.  To some extent, this situation is perhaps because, from the mathematical aspect, one needs to go beyond the well developed theory of one-parameter semigroups and deal with the so-called reverse evolution families. One of the difficulties arising in this case is related to the regularity w.r.t. time:
for one-parameter compositional semigroups, continuity in time often implies differentiability or even real analyticity, but such is not the case for evolution families.

The main purpose of this paper is to apply to the study  of time-inhomogeneous one-dimensional branching processes recent achievements on a Carath\'eodory-type ODE for \hbox{(reverse)} evolution families of holomorphic functions  developed in the frames of Loewner Theory since around 2010; see e.g. \cite{BCM1,BRFP2015,CDG_decr}. We mainly work with the continuous state space $[0,\infty]$ but also discuss the discrete states $\{0,1,2,\dots\}\cup\{\infty\}$.   In fact, the idea of using reverse evolution families to analyze time-inhomogeneous branching processes appeared for the case of discrete state space in Gorya\u{\i}nov~\cite{Goryainov96R, Goryainov96}, see also \cite[p.\,1012]{Goryainov-survey}. We will more systematically develop this idea, especially for the continuous state space $[0,\infty]$. In particular, we are aiming at finding interpretations of some relevant notions in Complex Analysis, and especially in Holomorphic Dynamics, from the viewpoint of Probability Theory and vice versa. Worth mentioning that this adds another application of Loewner Theory in Probabilities; other previously known applications are mostly limited to SLE, see e.g.~\cite{Lawler}, and certain aspects of non-commutative probability \cite{Bauer,FHS20,Jek20}.

The work has been divided into two parts. In the first part~\cite{GHP} we have prepared complex-analytic tools related to reverse evolution families, which are applied in the second part, i.e.\ in the present paper, to time-inhomogeneous branching processes. In particular, we show that under a very general assumption on regularity w.r.t.~$s$ and~$t$, the Laplace exponents $(v_{s,t})$ of a given time-inhomogeneous continuous state branching process satisfy a special  form of the Loewner\,--\,Kufarev ODE driven by a time-varying branching mechanism~$\phi(\cdot,s)$; see equation\,\eqref{eq:ODE}. In a sense, this is complementary to recent results of Fang and Li. In~\cite{FL}, they construct a continuous state branching process with the Laplace exponents defined as solutions to an integral form of the ODE~\eqref{eq:ODE} with identically vanishing free term in the Silverstein-type representation~\eqref{eq:BF_generator} of the branching mechanism. While considering the integral equation would permit to relax the regularity assumptions, working with the Loewner\,--\,Kufarev ODE enables us to benefit from Loewner Theory, with the help of which we establish a number of new results. Most important of them can be briefly summarised as follows:

\medskip

\begin{description}
\item[Theorem~\ref{thm:REF_BP}] assuming stochastic continuity by definition, we establish a one-to-one correspondence between transition kernels~$(k_{s,t})$ of branching processes and topological reverse evolution families~$(v_{s,t})$ consisting of Bernstein functions;
    \smallskip

\item[Theorem~\ref{thm:main_BP}] a one-to-one correspondence is established between absolutely continuous reverse evolution families~$(v_{s,t})$ consisting of Bernstein functions and Herglotz vector fields~$\phi$ admitting the Silverstein-type integral representation~\eqref{eq:BF_generator};
    \smallskip

\item[Theorem~\ref{TH_1-to-1-ACloc}] we prove that the first two moments of a continuous state branching process~$(Z_t)$ are finite and locally absolutely continuous (as functions of~$t$) if and only if the reverse evolution family of Laplace exponents~$(v_{s,t})$ associated to~$(Z_t)$ is absolutely continuous and the parameters in the Silverstein-type representation~\eqref{eq:BF_generator} of its Herglotz vector field~$\phi$ satisfy two simple conditions;
    \smallskip
   
\item[Theorem~\ref{thm:BP-for-DWzero}, Corollaries~\ref{cor:DW_infty} and~\ref{cor:DW_finite}] Silverstein-type representation formulas for Herglotz vector fields in case of the Laplace exponents with the Denjoy\,--\,Wolff point\footnote{For the definition of the Denjoy\,--\,Wolff point, see Appendix~\ref{App-Fixed} or \cite[Sect.\,2.3]{GHP}.} located at the same point $\tau\in[0,\infty]$;
    \smallskip
   
\item[Theorem~\ref{TH_infty-to-zero} and Corollary~\ref{TH_infty-to-nonzero}] sufficient conditions for almost sure extinction within a finite time and for almost surely \textit{no} extinction within a finite time  in cases of the Denjoy\,--\,Wolff point at ${\tau=0}$ and at ${\tau=\infty}$, respectively;
    \smallskip
   
\item[Theorem~\ref{thm:non-decreasing}] a probabilistic interpretation of the Denjoy\,--\,Wolff point at~$\infty$;
    \smallskip
   
\item[Theorem~\ref{TH_infty-to-tau}] probabilities of extinction and explosion within a finite time in case of a common interior Denjoy\,--\,Wolff point~${\tau\in(0,\infty)}$;
    \smallskip
   
\item[Section~\ref{sec3}]\hspace{-.37em} Analogous results are obtained in Sect.~\ref{sec3} for branching processes on the discrete state space;%
    \smallskip
   
\item[Theorem~\ref{thm:non-decreasingD}] we show that a branching process~$(N_t)$ on the discrete state space can be embedded into a branching process~$(Z_t)$ on the continuous state space if and only if $(N_t)$ is ``non-decreasing'' in a suitable sense.
\end{description}

\section{Preliminaries}\label{S_preliminaries} In this section we collect some basic definitions and preparatory facts in probability theory and complex analysis, mostly taken from the literature. Some material intended to make this paper more self-contained can be found in the Appendices. Our main results will be stated and proved in Sections~\ref{sec2} and~\ref{sec3}.
\subsection{One-parameter semigroups of holomorphic self-maps}\label{SS_one-param}
For a domain ${D\subset\C}$ and a non-empty set ${E\subset\C}$, denote by $\Hol(D,E)$ the set of all holomorphic maps of $D$ into~$E$. We make $\Hol(D,E)$ a topological space by endowing it with locally uniform convergence in~$D$.

A \textsl{semigroup of holomorphic self-maps} of  a domain $D\subset \C$ is a subset $\U\subset\Hol(D,D)$ containing the identity map~$\id_D$ and such that $f\circ g\in\U$ for any~$f,g\in\U$.
We say that such a semigroup $\U$ is \textsl{topologically closed} if $\U$ is a closed subset of $\Hol(D,D)$.

%\begin{definition}
Let $\U$ be a semigroup of holomorphic self-maps of a domain ${D\subset\C}$. A family ${(v_t)_{t\ge0}\subset\U}$ is said to be
a \textsl{one-parameter semigroup} in~$\U$, if the map ${t\mapsto v_t}$ is a continuous unital semigroup homomorphism from the semigroup ${\big([0,+\infty),+\big)}$ to the semigroup~${(\U,\circ)}$.
%\end{definition}
%

\smallskip\noindent{\bf Convention.~}~From now on we suppose that $D$ is either the unit disk~$\UD:=\{z\colon|z|<1\}$ or the right half-plane $\UH:=\{\zeta:\Re \zeta>0\}$.\smallskip

According to a result of Berkson and Porta \cite{BP78}, a one-parameter semigroup $(v_t)_{t\ge0}$ in $\Hol(D,D)$ is differentiable in $t \ge0$ and  can be obtained by a unique holomorphic function (called the \textsl{infinitesimal generator}) $\phi\colon D \to \mathbb C$ via the ODE
\begin{equation}\label{eq:ODEAuto}
\frac{\di}{\di t} v_t(z) = -\phi(v_t(z)),\qquad t\ge0, ~z\in D.
\end{equation}

For a semigroup $\U$ of holomorphic self-maps of $D$, we will denote by $\Gen(\U)$ the set of all infinitesimal generators of one-parameter semigroups in $\U$.
According to \cite[Theorem~1]{GHP}, if $\U$ is topologically closed, then $\Gen(\U)$ is a closed real convex cone in $\Hol(D,\C)$.

\subsection{Bernstein functions}\label{subsec:Bernstein}
A function $f\colon (0,\infty) \to \R$ is called a \textsl{Bernstein function} if it is of class $C^\infty$, $f\ge0$, and $(-1)^{n-1} f^{(n)}\ge0$ for all $n \in \N$. Thanks to monotonicity of $f^{(n)}$, there exist (finite or infinite) limits
\begin{equation}\label{EQ_BF-limits}
 f^{(n)}(0)\coloneqq \lim_{\theta\to 0^+}f^{(n)}(\theta) \quad \text{and} \quad  f^{(n)}(\infty)\coloneqq \lim_{\theta\to \infty}f^{(n)}(\theta),\quad \quad n\in\N\cup\{0\}.
\end{equation}
Note that the value $f(0)$ is finite because $f$ is non-decreasing and non-negative.

A detailed theory of Bernstein functions can be found, e.g., in monograph~~\cite{SSV12}. A fact of crucial importance for the present study is that every Bernstein function extends to a holomorphic map from the right half-plane $\UH\coloneqq \{\zeta \in \C: \Re \zeta>0\}$ into its closure $\overline{\UH}$. From now on, a holomorphic function in $\Hol(\UH,\overline{\UH})$ whose restriction to $(0,\infty)$ is  a Bernstein function will be also referred to as a Bernstein function.

Applying the maximum principle to the harmonic function $-\Re f$, it is easy to show that if a Bernstein function $f$ is not identically zero, then it is a self-map of~$\UH$.
Therefore, the set $\BF$ of all Bernstein functions $f\not\equiv0$ is a subset of $\Hol(\UH,\UH)$.
Another cornerstone fact is that $\BF$ is a topologically closed semigroup in $\UH$.

The class $\BF$ has been a basic tool for studying time-homogeneous branching processes; see \cite{Gre74,Kyp14,Sil68}.
Among others, Silverstein \cite[Theorem 4]{Sil68}  identified $\Gen(\BF)$, the set of all infinitesimal generators of the one-parameter semigroups in $\BF$, with the set of functions
\begin{equation}\label{gen}
\phi(\zeta) = -q+ a \zeta + b \zeta^2 + \int_0^\infty (e^{-\zeta x} - 1 + \zeta x \ind_{(0,1)}(x)) \,\pi(\di x), \qquad  \zeta\in \UH,
\end{equation}
where $a \in \R,~ q, b \ge0$, and $\pi$ is a \textsl{L\'evy measure}, i.e.\ a non-negative Borel measure on ${(0,\infty)}$ satisfying
\begin{equation}\label{EQ_Levy-measure}
  \int_0^\infty \min\{x^2,1\} \,\pi(\di x)<\infty.
\end{equation}

Note that the quadruple $(q,a,b,\pi)$ is unique for each $\phi$ and hence this quadruple completely parameterizes the set $\Gen(\BF)$. Other proofs of the representation~\eqref{gen} are available in \cite{CLUB09} and~\cite[Section~3.3, Appendix~D]{GHP}.

One of the main achievements in this paper is that we are able to extend Silverstein's theorem to the time-inhomogeneous setting, see Section~\ref{subsec:ODE}. This requires complex-analytic tools extending the theory of one-parameter semigroups to a ``non-autonomous'' setting, as we describe below.

\subsection{Reverse evolution families}\label{SS_REF}
A non-autonomous version of the ODE~\eqref{eq:ODEAuto}, i.e.\ an equation of the same form but with the vector field $\phi$ depending explicitly on~$t$, has been intensively studied in Complex Analysis. One-parameter semigroups of holomorphic self-maps in this setting are replaced by more general (reverse) evolution families.

For later use, denote by $\Delta$ the set $\{(s,t)\colon s,t\in [0,\infty),\,s\le t\}$
and let $\Lloc$ stand for the class of all measurable (real- or complex-valued) functions on $[0,\infty)$ which are integrable with respect to the Lebesgue measure on each compact subinterval of~$[0,\infty)$.

The following definitions are slight modifications of the definitions given in \cite{BCM1,CDG_decr}.

\begin{definition}\label{def:HVF}
Let \DCn. A function ${\phi:D\times [0,\infty)\to\C}$ is said to be a \textsl{Herglotz vector field} in~$D$ if it satisfies the following three conditions:
\begin{enumerate}[label=\rm(HVF\arabic*),leftmargin=20mm]
\item\label{HVF1} for any $t\ge0$, $\phi(\cdot,t)$ is an infinitesimal generator in~$D$;
\item\label{HVF2} for any $z\in D$, $\phi(z,\cdot)$ is measurable on~$[0,\infty)$;
\item\label{HVF3} for any compact set $K\subset D$, there is a non-negative function ${M_K\in\Lloc}$ such that
\[
\max_{z\in K}|\phi(z,t)|\le M_K(t)\quad\text{a.e.~$t\in [0,\infty)$}.
\]
\end{enumerate}%\smallskip

As usual we identify two Herglotz vector fields ${\phi,\psi\colon D\times [0,\infty)\to\C}$  if $\phi(\cdot,t)=\psi(\cdot,t)$ for a.e.~${t\in [0,\infty)}$.
The notion of reverse evolution families, introduced below, provides a way to describe solutions to the ODE driven by a Herglotz vector field.
\end{definition}

\begin{definition}\label{DF_REF}
An \textsl{absolutely continuous reverse evolution family} in \DCn is a two-parameter family ${(v_{s,t})_{(s,t)\in \Delta}\subset\Hol(D,D)}$ satisfying the following conditions:
\begin{enumerate}[label=\rm(REF\arabic*),leftmargin=20mm]
\item\label{REF1} $v_{s,s}=\id_{D}$ for any~$s\ge 0 $,

\item \label{REF2} $v_{s,u}=v_{s,t}\circ v_{t,u}$ for any ${0\le s\le t \le u}$,

\item\label{REF3} For each $z\in D$ there exists a non-negative function $f_{z}\in\Lloc$ such that
\[
|v_{s,u}(z)-v_{s,t}(z)|\le\int_t^u\!f_{z}(r)\,\di r\qquad\text{for any ${0\le s\le t \le u}$.}
\]
\end{enumerate}

Furthermore, a two-parameter family ${(v_{s,t})_{(s,t)\in \Delta}\subset\Hol(D,D)}$ is called a \textsl{topological\\ reverse evolution family} if it satisfies \ref{REF1}, \ref{REF2} and
\begin{enumerate}[label=\rm(REF\arabic*${}'$),leftmargin=20mm]
\setcounter{enumi}{2}
\item \label{REF3'} $\Delta\ni(s,t)\mapsto v_{s,t} \in \Hol(D,D)$ is continuous.
\end{enumerate}
\end{definition}
Any absolutely continuous reverse evolution family satisfies \ref{REF3'} and hence it is also a topological reverse evolution family; this is a consequence of \cite[Proposition~3.5]{BCM1} and \cite[Proposition~4.3]{CDG_decr}.

\smallskip
From the viewpoint of dynamics in the complex plane, it is more natural to consider the condition $v_{t,u} \circ v_{s,t}= v_{s,u}~(0\le s \le t \le u)$ rather than \ref{REF2}; such a family is called an (absolutely continuous / topological) \textsl{evolution family}.  On the other hand, in the existing applications to (non-commutative) probability theory, reverse evolution families have been more natural; see \cite{FHS20,Lawler}. In the present work where branching processes are analyzed, again reverse evolution families turn out to be the natural objects.

Some of the main results of~\cite{CDG_decr} can be summarized and stated in our notation as follows.
\begin{result}\label{TH_mainBCM1}
Let $(v_{s,t})_{(s,t)\in\Delta}$ be an absolutely continuous reverse evolution family in~\DCn. Then there exists a unique Herglotz vector field ${\phi\colon D\times [0,\infty)\to\C}$ such that for any $z\in D$ and any $t\in (0,\infty)$, the map ${[0,t] \ni s\mapsto v(s)\coloneqq v_{s,t}(z)\in D}$ is a solution to the initial value problem
\begin{equation}\label{EQ_ini-LK-ODE}
\frac{\di v}{\di s}(s)=\phi(v(s),s),\quad \text{a.e.\ } s\in [0,t];\qquad v(t)=z.
\end{equation}

Conversely, let ${\phi\colon D\times [0,\infty)\to\C}$ be a Herglotz vector field. Then for any ${z\in D}$ and any $t\in (0,\infty)$ the initial value problem~\eqref{EQ_ini-LK-ODE} has a unique solution ${s\mapsto v(s)=v(s;z,t)}$ defined on some open interval of~$\Real$ containing~$[0,t]$. Setting ${v_{s,t}(z)\coloneqq v(s;z,t)}$ for all ${z\in D}$ and all $(s,t)\in\Delta$ one obtains an absolutely continuous reverse evolution family $(v_{s,t})_{(s,t)\in\Delta}$.
\end{result}
As clear from the definition of a Herglotz vector field, the r.h.s. of the equation \eqref{EQ_ini-LK-ODE} is holomorphic and hence locally Lipschitz in the ``spatial'' complex variable~$z$, but it is not supposed to be even continuous in the temporal variable~$s$.
This equation should be, therefore, understood as a Carath\'eodory-type ODE; see e.g.\ \cite[Section~2]{CDG-AnnulusI} for the basic theory of such~ODEs.
It is further worth mentioning that the ODE~\eqref{EQ_ini-LK-ODE} can be replaced by the following PDE, see~\cite[Remark~B.1]{GHP}:
\begin{equation}\label{EQ_ini-revLK-PDE}
\frac{\partial v_{s,t}(z)}{\partial t}+\phi(z,t)\frac{\partial v_{s,t}(z)}{\partial z}=0,\quad \text{a.e.\ } t\in  [s,\infty),~z\in D;\qquad v_{s,s}=\id_D.
\end{equation}
In Loewner Theory, \eqref{EQ_ini-LK-ODE} and \eqref{EQ_ini-revLK-PDE} are known as the (generalized) Loewner\,--\,Kufarev differential equations.
In the context of branching processes, see Theorems~\ref{thm:main_BP} and~\ref{thm:main_BP_D}, equations \eqref{EQ_ini-LK-ODE} and \eqref{EQ_ini-revLK-PDE} correspond to \textsl{Kolmogorov's  backward and forward equations}, respectively.

\smallskip\noindent{\bf Convention.~}~
In equation~\eqref{EQ_ini-revLK-PDE} and everywhere in what follows, ``a.e.'' means that the relation holds except for a null set on the time axis  {\fontseries{bx}\selectfont\textit{not depending}} on the complex variable(s).\smallskip

For more details on the dynamical properties of holomorphic self-maps of complex domains we refer interested readers to an older book~\cite{Abate:book} and two recent monographs~\cite{Abate2,BCD-Book}. The modern version of Loewner Theory used in this paper is covered in~\cite{BCM1,CDG_LCh,CDG_decr}. A more condensed account, following the same notations and conventions as in the present paper, can be found in~\cite{GHP}. Finally, to make the exposition self-contained, we include some material on boundary  fixed points of holomorphic self-maps in Appendix~\ref{App-Fixed}.

\subsection{Branching processes}\label{SS_branching-processes}
Let $S$ be a Polish space. Denote by $\mathcal B(S)$ the $\sigma$-field of all Borel sets in~$S$ and by $\Pr(S)$ the set of all Borel probability measures on~$S$.
A \textsl{transition kernel} $k$ on $S$ is a map $k\colon S \times \mathcal{B}(S) \to [0,1]$ such that $k(x, \cdot) \in \Pr(S)$ for each $x \in S$ and  $x\mapsto k(x,B)$ is Borel measurable for each $B \in \mathcal{B}(S)$. For two transition kernels $k,l$ on $S$ their \textsl{composition} is defined by
\[
(k\starM  l)(x,B) = \int_S l(y,B) k(x, \di y) , \qquad x \in S, ~B \in  \mathcal{B}(S).
\]
Note that for all $x \in S$ and all non-negative Borel measurable functions $f\colon S \to [0,\infty]$,
\begin{equation}\label{eq:fubini}
\int_S f(z) \,(k\starM  l)(x,\di z) = \int_S \left[ \int_S f(z)\, l(y,\di z) \right] k(x, \di y).
\end{equation}

\smallskip
Let $(k_{s,t})_{(s,t)\in\Delta}$ be a family of transition kernels such that
\begin{enumerate}[label=\rm(K\arabic*)]
\item\label{k1} $k_{s,s}(x,\cdot)= \delta_x(\cdot)$ for every $s \ge 0$ and $x \in S$,~\,~ and
\item \label{CK} $k_{s,t}\starM k_{t,u}=k_{s,u}$ for every $0 \le s \le t \le u$. \quad (Chapman\,--\,Kolmogorov equation)
\end{enumerate}
Then a \emph{Markov family} on $S$ with transition kernels $(k_{s,t})_{(s,t)\in\Delta}$ is a stochastic process $(Z_t)_{t\ge 0}$ which is defined on a measure space $(\Omega, \mathcal F)$ equipped with a family $$\{\mathcal F_I: I \text{~is a closed interval of $[0,\infty)$}\}$$ of sub-$\sigma$-fields  of $\mathcal F$ and probability measures $\mathbb P^{(s,x)}$ on $\big(\Omega, \mathcal F_{[s,\infty)}\big)$ for each~${s \ge0}$~and~${x \in S}$ satisfying the following four conditions:
 $~\,\mathcal F_I \subseteq \mathcal F_J$ if $I \subseteq J$,  $~\,Z_t$ is $\mathcal F_{I}$-measurable if~${t \in I}$,
\begin{align}
&\mathbb P^{(s,x)}\big[Z_s=x\big]=1  \qquad \text{for all~}s\ge0,~x\in S,\qquad \text{and}  \label{eq:initial} \\
&\mathbb E^{(s,x)}\big[f(Z_u) \big| \mathcal F_{[s,t]}\big] = \int_S f(y) \,k_{t,u}(Z_t, \di y) \qquad \mathbb P^{(s,x)}\text{\,-\,a.s.}  \label{eq:Markov}
\end{align}
for all~$0 \le s \le t \le u$, $x \in S$, and all bounded measurable functions ${f\colon\! (S, \mathcal B(S))\to (\R,\mathcal B(\R\mathrlap{)).}}$
\smallskip

When $k_{s,t}=k_{0,t-s}$ for all ${(s,t)\in\Delta}$, the transition kernels and the corresponding Markov family are said to be \textsl{time-homogeneous}.\smallskip

Formulas \eqref{eq:initial} and \eqref{eq:Markov} imply the identity
\begin{equation}\label{eq:FD}
\mathbb P^{(s,x)}[(Z_{t_1}, Z_{t_2},\dots, Z_{t_n}) \in B]=\int_{B} k_{s,t_1}(x, \di x_1) k_{t_1,t_2}(x_1, \di x_2) \cdots k_{t_{n-1}, t_n}(x_{n-1}, \di x_n)
\end{equation}
for all $B \in \mathcal B (S^n), n \in \N, 0 \le s \le  t_1 \le \cdots \le t_n$ and $x \in S$, see \cite[Problem 8.11]{Wen81}. The l.h.s. of~\eqref{eq:FD} is called a finite dimensional distribution of $(Z_t)$. For further details on Markov families, we refer the reader to \cite{Wen81}.
\begin{definition}
By a \textsl{time-inhomogeneous continuous-state branching process} (simply called a \textsl{branching process} below) we mean a Markov family on $S:=[0,\infty]$ (regarded as a compact set) with transition kernels $(k_{s,t})_{(s,t)\in\Delta}$ satisfying \ref{k1} and \ref{CK} above as well as the following four conditions:
\begin{enumerate} [label=\rm(K\arabic*)]
\setcounter{enumi}{2}
\item\label{k3} the map $\Delta \times [0,\infty)\ni (s,t,x)\mapsto k_{s,t}(x,\cdot) \in \Pr([0,\infty])$ is weakly continuous,
\item \label{eq:branching} $k_{s,t} (x, \cdot) \ast k_{s,t}(y, \cdot) = k_{s,t}(x+y,\cdot)$ for any~$(s,t)\in\Delta$, $x,y \in [0,\infty),$ where as usual, we denote by~$\ast$ the convolution of probability measures on $[0,\infty]$,
\item\label{k5} $k_{s,t}(0,\cdot)= \delta_0(\cdot)$ for every $(s,t)\in\Delta$,
\item\label{k6} $k_{s,t}(\infty, \cdot) = \delta_\infty(\cdot)$ for every $(s,t)\in\Delta$.
\end{enumerate}
Such a family $(k_{s,t})_{(s,t)\in\Delta}$ on $[0,\infty]$ that satisfies \ref{k1}--\ref{k6} will be called a \textsl{family of transition kernels of a branching process}. Condition \ref{eq:branching} is called \textsl{branching property}.
\end{definition}
\begin{remark}\label{RM_branching-process-from-trans-probas}
Given a family of transition kernels~$(k_{s,t})$ satisfying \ref{k1} and \ref{CK}, one can construct a corresponding Markov family in a canonical way. Let $\Omega:=S^{[0,\infty)}$ be the set of all functions from $[0,\infty)$ to $S$ and $(Z_t)_{t\ge0}$ be the coordinate process on~$\Omega$, i.e.\ ${Z_t(\omega):=\omega(t)}$, ${t\ge0}$, ${\omega \in \Omega}$.  Let $\mathcal F$ and $\mathcal F_{I}$ be the $\sigma$-fields generated by the sets $\{Z_t^{-1}(B)\colon B \in \mathcal B(S), t\ge0 \}$ and  $\{{Z_t^{-1}(B)}\colon {B \in \mathcal B(S)},~{t \in I}\}$, respectively. Then one can assign to each $(s,x)$ a probability measure $\mathbb P^{(s,x)}$ on $(\Omega, \mathcal F_{[s,\infty)})$ in such a way that the coordinate process $(Z_t)_{t\ge0}$ becomes a Markov family with the transition kernels~$(k_{s,t})$, see \cite[Theorem 8.2]{Wen81}. If furthermore the transition kernels~$(k_{s,t})$ satisfy \ref{k3}--\ref{k6} then $(Z_t)$ is a branching process by definition.
\end{remark}

\section{Time-inhomogeneous branching processes}\label{sec2}
In this section, we will establish our main results concerning time-inhomogeneous branching processes on a continuous one-dimensional state space.

Properties \ref{k5} and \ref{k6} in the definition of a time-inhomogeneous continuous-state branching process, adopted in Section~\ref{SS_branching-processes}, intuitively mean that $0$ and $\infty$ are absorbing points for the process.
Properties~\ref{k3}, \ref{eq:branching},  \ref{k5} and~\ref{k6} imply that each probability measure $k_{s,t}(x,\cdot)$ is the law of a killed subordinator at ``time'' $x$, see  e.g.\ \cite[p.\,48]{SSV12}. Therefore, see e.g.\ \cite[Theorem 5.2]{SSV12}, for any $(s,t)\in\Delta$ there exists a unique Bernstein function ${v_{s,t}\colon(0,\infty) \to [0,\infty)}$, called the \textsl{Laplace exponent} of~$k_{s,t}$, such that
\begin{equation}\label{eq:laplace}
\LpT{k_{s,t}(x,\cdot)}(\theta):=\int_{[0,\infty)} e^{-\theta y} k_{s,t}(x, \di y) = e^{-x v_{s,t}(\theta)}, \qquad x,\theta \in (0,\infty).
\end{equation}
Recall that any Bernstein function extends to a self-map of $[0,\infty)$; then the above formula can be extended to $\theta \in [0,\infty)$ by the monotone convergence theorem.

\subsection{Branching processes and reverse evolution families}\label{SS_topological}
The transition kernels of branching processes are characterized by $v_{s,t}$'s as follows.

\begin{theorem}\label{thm:REF_BP}  Given a family $(k_{s,t})_{(s,t)\in\Delta}$ of transition kernels of a branching process, the Laplace exponents $v_{s,t}$ defined by \eqref{eq:laplace} form a topological reverse evolution family~$(v_{s,t})_{(s,t)\in\Delta}$ contained in $\BF$.
Conversely, given a topological reverse evolution family ${(v_{s,t})_{(s,t)\in\Delta}\subset\BF}$, there exists a unique family of
transition kernels $(k_{s,t})_{(s,t)\in \Delta}$ of a branching process such that \eqref{eq:laplace} holds.
\end{theorem}
\begin{proof}
We divide the proof in three steps.

\Step{1}
First we suppose that an arbitrary family $(k_{s,t})_{(s,t)\in\Delta}$ of transition kernels on~$[0,\infty]$, not necessarily related to  a branching process, and a family $(v_{s,t})_{(s,t)\in\Delta}$ of maps from $(0,\infty)$ to $[0,\infty)$ are given and related by \eqref{eq:laplace}, i.e.\ for each $x>0$ and $(s,t)\in\Delta$, the function $\exp(-x v_{s,t}(\cdot))$ is the Laplace transform~$\LpT{k_{s,t}(x,\cdot)}$ of the measure $k_{s,t}(x,\cdot)$. In particular, $v_{s,t}$ is non-decreasing and hence extending $v_{s,t}$ to $\theta=0$ by $v_{s,t}(0):=\lim_{\theta \to0^+}v_{s,t}(\theta)$, we may assume that $v_{s,t}$ is a self-map of $[0,\infty)$ and that \eqref{eq:laplace} holds for all~${\theta \ge0}$. Under this circumstance we prove that the following three conditions:
\begin{enumerate}[label=\rm(C\arabic*)]
\item\label{R1} $v_{s,s}=\id_{[0,\infty)}$ for any~$s\ge 0 $,

\item \label{R2} $v_{s,u}=v_{s,t}\circ v_{t,u}$ on $[0,\infty)$ for any ${0\le s\le t \le u}$,

\item \label{R3} for every $\theta \in (0,\infty)$, the map $\Delta\ni(s,t)\mapsto v_{s,t}(\theta) \in [0,\infty)$ is continuous,
\end{enumerate}
 are equivalent to \ref{k1}--\ref{k3} for $(k_{s,t})_{(s,t)\in\Delta}$.  Recall that a bounded Borel measure on $[0,\infty)$ is uniquely determined by its Laplace transform, see e.g.~\cite[Proposition~1.2]{SSV12}. Hence, it is clear that \ref{R1} is equivalent to \ref{k1}. Furthermore, \ref{R2} is equivalent to the Chapman\,--\,Kolmogorov relation~\ref{CK},  because thanks to \eqref{eq:fubini} we have
\begin{align*}
e^{-x v_{s,t}(v_{t,u}(\theta))}
&=  \int_{[0,\infty)} e^{-  z v_{t,u}(\theta)} k_{s,t}(x,\di z)  \\
&=   \int_{[0,\infty)} \left( \int_{[0,\infty)}e^{-  \theta y} k_{t,u}(z,\di y)\right) k_{s,t}(x,\di z)
=  \int_{[0,\infty)} e^{-  \theta y} (k_{s,t}\starM k_{t,u})(x,\di y).
\end{align*}
The implication \ref{k3} $\Rightarrow$ \ref{R3} is rather obvious from the definition of weak convergence: indeed, the map $y\mapsto e^{-\theta y}$ can be regarded as an element of $C[0,\infty]$ vanishing at $\infty$ as long as $\theta >0$, and hence, we can write $e^{-x v_{s,t}(\theta)} =\int_{[0,\infty]} e^{-  \theta y} k_{s,t}(x, \di y) $. For the converse implication, one can use the fact that the pointwise convergence of the Laplace transforms $\LpT{k_{s,t}(x,\cdot)}$ on $(0,+\infty)$ as $(s,t,x)\to(s_0,t_0,x_0) \in \Delta \times [0,\infty)$  implies the vague convergence $k_{s,t}(x,\cdot) \to k_{s_0,t_0}(x_0,\cdot)$ as \emph{sub-probability measures on $[0,\infty)$}, see e.g.\ \cite[Lemma~A.9]{SSV12}. This  further implies the weak convergence $k_{s,t}(x,\cdot) \to k_{s_0,t_0}(x_0,\cdot)$ as \emph{probability measures on $[0,\infty]$}, because for every $f\in C[0,\infty]$,
\[
\int_{[0,\infty]} f(y) \,k_{s,t}(x,\di y) = \int_{[0,\infty)} [f(y)-f(\infty)] \,k_{s,t}(x,\di y) +f(\infty),
\]
and $f-f(\infty)$ is a continuous function on $[0,\infty]$ vanishing at $\infty$, and hence the integral in the r.h.s.\ above converges to $\int_{[0,\infty)} [f(y)-f(\infty)] \,k_{s_0,t_0}(x_0,\di y)$ as $(s,t,x)\to(s_0,t_0,x_0)$ by the vague convergence of $k_{s,t}(x,\cdot)$.

\Step{2}
Now we suppose that $(k_{s,t})_{(s,t)\in\Delta}$ is a family of transition kernels of a branching process, i.e.\ the conditions \ref{k1}--\ref{k6} hold.  Then as explained at the beginning of Section~\ref{sec2}, there exists a family $(v_{s,t})_{(s,t)\in\Delta}$ of Bernstein functions satisfying~\eqref{eq:laplace}. As we have proved in \StepN{1}, \ref{k1}--\ref{k3} imply that conditions \ref{R1}--\ref{R3} hold.

Each of the Bernstein functions $v_{s,t}$ either extends to a holomorphic self-map of~$\UH$ or vanishes identically, see Section~\ref{subsec:Bernstein}.  Let us show that all $v_{s,t}$'s extend to self-maps of~$\UH$.
Suppose to the contrary that ${v_{s_0,t_0}\equiv0}$  for some ${0 \le s_0 <t_0}$. Let  $t_1\coloneqq  \min\{t \in(s_0,t_0]: v_{s_0,t} \equiv 0\}$, where the minimum exists thanks to the continuity~\ref{R3}.
Then for any ${t\in[s_0,t_1)}$ the function $v_{s_0,t}$ does not vanish in~$\UH$ and as a result, the relation $v_{s_0, t} \circ v_{t ,t_1}={v_{s_0,t_1}\equiv 0}$ on $[0,\infty)$ implies that $v_{t, t_1}\equiv0$.
Passing to the limit as $t\to t_1^-$ yields that $v_{t_1,t_1}\equiv0$, a contradiction. Therefore, one concludes that all $v_{s,t}$'s are holomorphic self-maps of~$\UH$  and hence $(v_{s,t})\subset\BF$.

Conditions \ref{R1} and \ref{R2} imply \ref{REF1} and \ref{REF2} by the identity principle for holomorphic functions. Condition~\ref{REF3'} follows from \ref{R3} due to Vitali's theorem, see e.g.\  \cite[\S\,I.2]{Goluzin}, applied to the functions $f_{s,t}:=H\circ v_{s,t}$, where ${H(z):=(z-1)/(z+1)}$ maps~$\UH$ conformally onto~$\UD$, and due to the uniqueness of the holomorphic extension. Thus, $(v_{s,t})$ is a topological reverse evolution family contained in~$\BF$.

\Step{3}
Conversely, given a topological reverse evolution family $(v_{s,t})_{(s,t)\in\Delta}$ contained in~$\BF$, a unique family $(k_{s,t}(x,\cdot))_{(s,t)\in\Delta, x \in [0,\infty]}$ of probability measures on $[0,\infty]$ satisfying \eqref{eq:laplace}, \ref{eq:branching}, \ref{k5}, and \ref{k6} does exist thanks to the theory of subordinators \cite[Theorem~5.2]{SSV12}. It remains to notice that by
 \StepN{1}, conditions \ref{REF1}, \ref{REF2}, and \ref{REF3'} in the definition of a topological reverse evolution family imply conditions \ref{k1}--\ref{k3}.
\end{proof}

The construction of branching processes with prescribed family of transition kernels~$(k_{s,t})$ given in Remark~\ref{RM_branching-process-from-trans-probas} is sometimes not enough. For example, when we later discuss hitting times to~$0$ and $\infty$, the monotonicity of events $\{Z_t =0\}$ w.r.t.\ $t\ge0$ is helpful. Standard technique of supermartingales enables to establish existence of a ``better'' branching process.
\newcommand{\Xc}{X^{\mathrm{c}}} Namely, for $s\in[0,\infty)$, denote by $D[s,\infty)$ be the set of all \cadlag functions ${\omega\colon[s,\infty) \to [0,\infty]}$ satisfying  the following condition: if $r:=\omega(t)\in \{0,\infty\}$ for some ${t\ge s}$ or $r:=\omega(t-)\in \{0,\infty\}$ for some ${t>s}$, then $\omega(u)=r$ for all~${u\ge t}$.

Consider the coordinate process $(\Xc_t)_{t\ge0}$ defined by $\Xc_t(\omega):=\omega(t)$, $\omega \in D[0,\infty)$, and let $(\mathcal G_I)_I$ stand for its natural filtration, i.e. $\mathcal G_I:=\sigma(\Xc_t: t \in I)$ for any closed interval ${I\subset [0,\infty)}$. Finally, let ${\mathcal G: = \mathcal G_{[0,\infty)}}$.

\begin{theorem}\label{thm:cadlag} Let $(k_{s,t})$ be a family of transition kernels satisfying conditions \hbox{\ref{k1}--\ref{k6}}. Then there exists a unique map $(s,x)\mapsto \mathbb Q^{(s,x)}$ assigning to each ${s \ge 0}$ and ${x\in[0,\infty]}$ a probability measure $\mathbb Q^{(s,x)}$ on the measurable space $\big(D[0,\infty),\mathcal G_{[s,\infty)}\big)$ such that
$\big((\Xc_t)_{t\ge0}, D[0,\infty), \mathcal G, (\mathcal G_I)_I, (\mathbb Q^{(s,x)})_{s\ge0,\, x \in [0,\infty]}\big)$ is a branching process on~$[0,\infty]$ with the transition kernels~$(k_{s,t})$.
\end{theorem}
\noindent The proof of Theorem~\ref{thm:cadlag} is rather standard; however, for the sake of completeness, we provide it in \hbox{Appendix~\ref{App-Th2.7}.}

\subsection{An \,ODE\, for Laplace exponents}\label{subsec:ODE}
Now we turn to examining branching processes which correspond to absolutely continuous reverse evolution families. We start by establishing a basic ODE. In order to give a precise statement of the result, we have to introduce some notation and definitions.

In the following, we employ the notation $f^{(n)}(z,t)$, $f'(z,t)$, $f''(z,t)$ and so on for the derivatives w.r.t. the complex variable~$z$, while keeping the notation $\,\frac{\partial }{\partial t} f(z,t)\,$ or $\,\partial_t f(z,t)$ for the derivatives w.r.t.\ the real parameter~$t$.
Recall that a family of non-negative Borel measures $(\rho_t)_{t\ge0}$ on $[0,\infty)$ is a \textsl{Borel kernel} if the function $[0,\infty) \ni t\mapsto\rho_t(B) \in [0,\infty]$ is measurable for every Borel subset $B$ of ${[0,\infty)}$. Moreover, we say that a Borel kernel~$(\rho_t)$ is \textsl{locally integrable (in $t$)} if ${t\mapsto \rho_t(B)}$ is in $\Lloc$ for every Borel subset~$B$ of ${[0,\infty)}$.

\begin{definition}
A \textsl{L\'evy family} is a family of quadruples $(q_t,a_t,b_t,\pi_t)_{t\ge0}$ in which
${t\mapsto a_t\in\Real}$, $t\mapsto b_t\in[0,\infty)$, $t\mapsto q_t\in[0,\infty)$ are  $\Lloc$-functions on ${[0,\infty)}$ and $(\pi_t)_{t\ge0}$ is a family of non-negative Borel measures such that  $t\mapsto \min\{x^2,1\}\,\pi_t(\di x)$ on $(0,\infty)$ is locally integrable in $t$.

We identify two L\'evy families $(q_t,a_t,b_t,\pi_t)_{t\ge0}$ and $(\tilde q_t,\tilde a_t, \tilde b_t,\tilde \pi_t)_{t\ge0}$ if there exists a Lebesgue null set $N \subset [0,\infty)$ such that $(q_t,a_t,b_t,\pi_t)=(\tilde q_t,\tilde a_t, \tilde b_t,\tilde \pi_t)$ for all $t\in [0,\infty)\setminus N$.
\end{definition}

Now we are able to formulate our main result asserting, in particular, that under the absolute continuity assumption, the Laplace exponents of a branching process satisfy a particular form of the Loewner\,--\,Kufarev ODE.

\begin{theorem}\label{thm:main_BP}
Let $(v_{s,t})_{(s,t)\in\Delta}$ be an absolutely continuous reverse evolution family contained in~$\BF$.
Then there exists a unique L\'evy family $(q_t,a_t,b_t,\pi_t)_{t\ge0}$ with the following property: for every $t \ge0$ and every $\zeta \in \UH,$ the map $s \mapsto v(s)\coloneqq v_{s,t}(\zeta)$ is the unique solution to the initial value problem
\begin{equation}\label{eq:ODE}
\frac{\di }{\di s} v(s)= \phi(v(s),s)  \quad \text{a.e.~}s \in [0,t]; \qquad v(t)=\zeta,
\end{equation}
where $\phi\colon \UH\times [0,\infty)\to\C$ is defined by
\begin{equation}\label{eq:BF_generator}
\phi(\zeta,t) := -q_t + a_t \zeta + b_t \zeta^2 + \!\int\nolimits_0^\infty\!\! \big(e^{-\zeta x} - 1 + \zeta x \ind_{(0,1)}(x)\big) \,\pi_t(\di x),~\,~ \zeta \in \UH, ~\, t\ge0,\hspace{-1em}
\end{equation}
i.e.\ the Herglotz vector field associated with $(v_{s,t})_{(s,t)\in\Delta}$ is of the form \eqref{eq:BF_generator}.

Conversely, given a L\'evy family $(q_t,a_t,b_t,\pi_t)_{t\ge0}$,  the function $\phi$ defined by \eqref{eq:BF_generator} is a Herglotz vector field, and for every $\zeta \in \UH$ and $t\ge0$, the initial value problem~\eqref{eq:ODE} has a unique solution $s\mapsto v(s)=v(s;t,\zeta)$.  Moreover, the formula ${v_{s,t}(\zeta)\coloneqq v(s;t,\zeta)}$ defines an absolutely continuous reverse evolution family $(v_{s,t})_{(s,t)\in\Delta}$  contained in $\BF$.
\end{theorem}

The second statement of the above theorem has recently appeared in the literature (see e.g. \cite[Theorems~6.17 and~6.18]{Li22}) under certain additional conditions on the L\'evy family. Below we show that Theorem~\ref{thm:main_BP} can be obtained as a direct corollary of general facts in Loewner Theory and \cite[Theorem~2]{GHP}. However, there is a non-trivial part of the argument, which is worth being stated separately. Namely, before deducing Theorem~\ref{thm:main_BP}, we are going to prove the following proposition.
Recall that a \textsl{L\'evy measure} is a non-negative Borel measure $\pi$ on~$(0,\infty)$ satisfying the integrability condition~\eqref{EQ_Levy-measure}.

\begin{proposition}\label{prop:int}
For each $t\ge0$, let $a_t \in \R, ~q_t, b_t \in [0,\infty)$ and $\pi_t$ a L\'evy measure. Let $\phi \colon \UH \times [0,\infty) \to \C$
be defined by~\eqref{eq:BF_generator}.
Then the following assertions are equivalent:
\begin{enumerate}[label=\rm(\roman*)]
\item\label{item:HVF}
 $\phi$ is a Herglotz vector field, i.e.\ for each $\zeta\in \UH$ the function $t\mapsto \phi(\zeta,t)$ is measurable, and for each compact subset $K$ of $\UH$ the function
$
t \mapsto \sup_{\zeta\in K} |\phi(\zeta,t)|
$
is in $\Lloc$.

\item\label{item:HVF2} For each $\theta > 0$ the function $t\mapsto \phi(\theta,t)$ is measurable, and in addition,  there exists $\theta_0>0$ such that
$
t \mapsto \phi^{(n)}(\theta_0,t)
$
is in $\Lloc$ for all $n\in \{0,1,2\}$.

\item\label{item:int} $(q_t,a_t,b_t,\pi_t)_{t\ge0}$ is a L\'evy family.
\end{enumerate}
\end{proposition}

In the proof we will make use of the following auxiliary result, which is quite expected and can be established by a standard monotone class argument. For further details, see e.g.~\cite[Theorem~1.23]{Li22} and its proof; see also \cite[Appendix~A.3]{Sharpe88}.

\begin{lemma} \label{lem:measurable}
Let $(m_t)_{t\ge0}$ be a family a non-negative Borel measure on $[0,\infty)$. Suppose that the Laplace transform of $m_t$, i.e.
\[
\LpT{m_t}(\lambda) \coloneqq\int_{[0,\infty)} e^{-\lambda x} \,m_t(\di x),
\]
is finite for every $t\ge0$ and $\lambda>0$. Then the following conditions are equivalent:
\begin{enumerate}[label=\rm(\arabic*)]
\item\label{m1} the function $t\mapsto \LpT{m_t}(\lambda)$ is measurable for each $\lambda>0;$
\item\label{m3} $(m_t)_{t\ge0}$ is a Borel kernel;
\item\label{m4} for any Borel measurable function $f\colon [0,\infty)\to [0,\infty]$,
 the family $(f m_t)_{t\ge0}$ is a Borel kernel, where $(f m_t)(\di x) := f(x) m_t(\di x)$.
\end{enumerate}
\end{lemma}
\begin{proof}[\proofof{Proposition~\ref{prop:int}}] \textsc{Step 1.} The implication \ref{item:int}~$\Rightarrow$~\ref{item:HVF} can be easily proved using the fact that
for any compact set $K\subset\UH$ there exists a constant $C_K>0$ such that
\begin{equation}\label{EQ_estimate-on-K}
\sup_{\zeta\in K}\big|e^{-\zeta x} - 1 + \zeta x \ind_{(0,1)}(x)\big| \le C_K \min\{1,x^2\}
                                                              \qquad \text{for all~}~x\ge0.
\end{equation}

\Step{2: \ref{item:HVF} $\Rightarrow$  \ref{item:HVF2}} This is a consequence of Cauchy's integral formula (for any $\theta_0>0$).

\Step{3: \ref{item:HVF2} $\Rightarrow$  \ref{item:int}} Note first that $\phi^{(n)}(\theta,t)$ are measurable in $t$ for all $\theta>0$ and $n \in \N\cup\{0\}$, as being pointwise limits of measurable functions.
Let $m_t(\di x)\coloneqq  x^2\,\pi_t(\di x) + 2b_t\delta_0$; it is a non-negative Borel measure on $[0,\infty)$ for each $t\ge0$ and its Laplace transform is finite on $(0,\infty)$.
The second $\theta$-derivative of $\phi$ is then given by
\[
 \phi''(\theta,t) = 2 b_t + \int_0^\infty x^2 e^{-\theta x} \,\pi_t(\di x) =   \int_{[0,\infty)} e^{-\theta x} \,m_t(\di x) = \LpT{m_t}(\theta).
\]
By Lemma~\ref{lem:measurable}, $(m_t)_{t\ge0}$ is a Borel kernel. Then $(x^2\,\pi_t(\di x))_{t\ge0}$ is a Borel kernel and $t\mapsto b_t$ is measurable.
We proceed bearing in mind Lemma~\ref{lem:measurable} but without mentioning it each time explicitly. In particular, we conclude that $(\min\{x^2,1\}\,\pi_t(\di x))_{t\ge0}$ is a Borel kernel.
Also, since $t\mapsto \phi''(\theta_0,t)$ is in $\Lloc$, both $t\mapsto b_t$ and $t\mapsto\int_0^\infty x^2 e^{-\theta_0x} \,\pi_t(\di x)$ are in $\Lloc$. The latter implies that
\begin{equation}\label{eq:int0}
t\mapsto\int_{(0,1)} x^2 \pi_t(\di x) \qquad \text{and} \qquad  t\mapsto\int_{[1,\infty)} xe^{-\theta_0x} \,\pi_t(\di x)
\end{equation}
are both in $\Lloc$. This further implies that
\begin{equation}\label{eq:int1}
 \int\limits_0^\infty x (-e^{-\theta_0x}+ \ind_{(0,1)}(x)) \,\pi_t(\di x) = \int\limits_{(0,1)} x (1-e^{-\theta_0x})\,\pi_t(\di x) - \int\limits_{[1,\infty)} xe^{-\theta_0x} \,\pi_t(\di x)
\end{equation}
 and
 \begin{equation}\label{eq:int2}
  \int_{(0,1)} (e^{-\theta_0x} - 1 + \theta_0x) \,\pi_t(\di x)
\end{equation}
are both in $\Lloc$ as functions of $t$.
From the formula
\[
\phi'(\theta_0,t) =  a_t +2 b_t \theta_0  +\int_0^\infty x (-e^{-\theta_0 x}+ \ind_{(0,1)}(x)) \, \pi_t(\di x)
\]
and the local integrability of $t\mapsto \phi'(\theta_0,t)$, of $t\mapsto b_t$ and of~\eqref{eq:int1}, it follows that $t\mapsto a_t$ is in~$\Lloc$.
Finally, from~\eqref{eq:BF_generator} together with the local integrability of $a_t,b_t, \phi(\theta_0,t)$ and of~\eqref{eq:int2}, it follows that the function
\[
t\mapsto q_t + \int_{[1,\infty)} (1-e^{-\theta_0x}) \pi_t(\di x)
\]
is in $\Lloc$, and hence, so are $t\mapsto q_t$ and $t\mapsto \pi_t([1,\infty))$. Recalling the local integrability of the first function of \eqref{eq:int0}, we conclude that
$
t\mapsto \int\nolimits_0^\infty \min\{x^2,1\}\, \pi_t(\di x)
$
is in $\Lloc$, which completes the proof of~\ref{item:int}.
\end{proof}

\begin{proof}[\proofofNP{Theorem~\ref{thm:main_BP}}] becomes immediate if we recall:
 \begin{enumerate}[label=\rm(\arabic*)]
 \item the one-to-one correspondence between absolutely continuous reverse evolution families $(v_{s,t})\subset\Hol(\UH,\UH)$ and Herglotz vector fields ${\phi:\UH\times[0,\infty)\to\C}$, see Theorem~\ref{TH_mainBCM1};
 \item the fact that by \cite[Corollary~3.3]{GHP}, ${(v_{s,t})\subset\BF}$ if and only if the corresponding Herglotz vector field~$\phi(\cdot,t)$ admits  for a.e.~${t\ge0}$ the Silverstein representation~\eqref{eq:BF_generator};
 \item Proposition~\ref{prop:int}, according to which a function represented by~\eqref{eq:BF_generator} is a Herglotz vector field if and only if the parameters in~\eqref{eq:BF_generator}, uniquely defined by~$\phi$ up to a Lebesgue null set on the time axis, form a L\'evy family~$(q_t,a_t,b_t,\pi_t)_{t\ge0}$. \qedhere
 \end{enumerate}
\end{proof}
\begin{remark}
The Bernstein functions $v_{s,t}$ in Theorem~\ref{thm:main_BP} extend continuously to the imaginary axis, see e.g. \cite[Proposition~3.6]{SSV12}. Moreover, $\phi(s,0):=\lim_{\theta\to0^+}{\phi(s,\theta)}$ exists and it is a non-positive real number for a.e.\ ${s\ge0}$, see e.g.\ \cite[Theorem~3]{GHP}. Using the fact that the estimate~\eqref{EQ_estimate-on-K} holds also for any compact~$K$ in the \textit{closed} half-plane ${\{\zeta:\Re\zeta\ge0\}}$, it is then possible to extend ODE~\eqref{eq:ODE} to the boundary point~${\zeta=0}$. Namely, $v(s):=v_{s,t}(0)$, ${s\in[0,t]}$, solves the initial value problem ${\di v/\di t=\phi\big(v(s),s\big)}$, ${v(t)=0}$. Note that in general, the solution is not unique. Uniqueness turns out to be related to the conservativeness (cf. \cite[Theorem~1.2]{FL} and \cite[Theorem~4.1]{LiLi2024}) which  we discuss below from a different perspective.
\end{remark}

\subsection{Branching processes with finite first moments}\label{subsec:first_moment}
Let $(Z_t)_{t\ge0}$ be a branching process whose Laplace exponents~$v_{s,t}$ are characterized as in Theorem~\ref{thm:main_BP}.
We say that $(Z_t)_{t \ge 0}$ is \textsl{conservative} if $\mathbb P^{(s,x)}[Z_t < \infty]=1$ for any $x>0$ and $(s,t)\in\Delta$. This is equivalent to $v_{0,t}(0)=0$ for all $t>0$ as follows from \eqref{eq:derivative_BP}  below.
According to \cite[Section~2]{Gre74} (or see e.g.\ \cite[Theorem~12.3]{Kyp14}), a necessary and sufficient condition for a \emph{time-homogeneous} branching process to be conservative is that
\begin{equation}\label{EQ_conserv-cond}
\int_{0}^\epsilon \frac{1}{|\phi(\theta)|}\,\di \theta =\infty,  \qquad \forall \epsilon >0,
\end{equation}
where $\phi$ is the branching mechanism of the process~$(Z_t)$, i.e.\ the infinitesimal generator of the one-parameter semigroup  $(v_t)$ formed by the Laplace exponents $v_t(\theta)=-\log \mathbb E^{(0,1)}[e^{-\theta Z_t}]$.

In fact, the above condition~\eqref{EQ_conserv-cond} can be easily deduced from a general result on the relationships between the boundary fixed points of a one-parameter semigroup and the boundary behaviour of its Koenigs function, see \cite{CoDiPo04} or \cite[\S13.6]{BCD-Book}. Unfortunately, there does not seem to be known any analogous results for (reverse) evolution families, unless we restrict consideration to boundary \emph{regular} fixed points (\textsl{BRFP} for short; see Appendix~\ref{App-Fixed} or \cite[Sect.\,2.3]{GHP} for the precise definition). As a result, in the time-inhomogeneous setting, the possibility to characterize conservative branching processes by a condition similar to~\eqref{EQ_conserv-cond} is rather unclear.

Instead, we discuss the condition of finite mean, which implies the conservativeness. Finite mean is commonly assumed for time-homogeneous branching processes~\cite{Li22} and can be well characterized. Even in the time-inhomogeneous case, we are able to completely characterize finite mean  in terms of Herglotz vector fields.
With $v_{s,t}(0)$, $v'_{s,t}(0)$, $\phi(0,s)$, $\phi'(0,s)$ understood as limits for~${\theta\to0^+}$, which exist by monotonicity,  our result can be stated as follows.

\begin{theorem}\label{thm:first_moment}
Let $(Z_t)_{t\ge0}$ be a branching process whose Laplace exponents~$(v_{s,t})$ are characterized by a Herglotz vector field $\phi$ as in Theorem~\ref{thm:main_BP}.
Then the following six conditions are equivalent:
\begin{enumerate}[label=\rm(\arabic*)]

\item\label{first_moment0} $\mathbb E^{(s,x)}[Z_t]<\infty$ for all $x \in (0,\infty)$ and ${(s,t)\in \Delta}$.

\item\label{first_moment1} $\mathbb E^{(0,1)}[Z_t]<\infty$ for all $t \ge 0$.

\item \label{first_moment1.5} $v_{0,t}(0)=0$ and $v'_{0,t}(0)<\infty$, i.e.\ $0$ is a BRFP of $v_{0,t}$ for all $t\ge 0$.

\item \label{first_moment2} $v_{s,t}(0)=0$ and $v'_{s,t}(0)<\infty$, i.e.\,$0$ is a BRFP of~${v_{s,t}}$~for~all~${(s,t)\in\Delta}$.

\item\label{first_moment3} $\phi(0,t)=0$ for a.e.\ $t\ge0$ and the function $t\mapsto \phi'(0,t)$ is in $\Lloc$.

\item \label{first_moment4} $q_t=0$ for a.e.\ $t\ge0$ and the function $t\mapsto \int_{[1,\infty)} x\,\pi_t(\di x)$ is in $\Lloc$.

\end{enumerate}
If one (and hence all) of the above conditions hold, then
\begin{equation}\label{EQ_math-exp}
\mathbb E^{(s,x)}[Z_t] = x v_{s,t}'(0)= x\exp \left(- \int_s^t \phi'(0,r) \,\di r \right)
\end{equation}
 for all $x \in (0,\infty)$ and ${(s,t) \in \Delta}$.
\end{theorem}

Note that Herglotz vector fields associated with reverse evolution families $(v_{s,t})\subset\BF$ satisfying condition~\ref{first_moment2} above can be characterized by a representation formula, see~\eqref{EQ_inf-gen-represent-DWzero} and Remark~\ref{RM-charac-BRFP-at-zero} in Section~\ref{subsub:0}.

\begin{proof}[\proofof{Theorem~\ref{thm:first_moment}}] We begin with establishing the equivalence of \ref{first_moment0}--\ref{first_moment2}.
Passing to the limit as $\theta \to0^+$ in the relation $\mathbb E^{(s,x)}[e^{-\theta Z_t} \ind_{\{Z_t <\infty\}}] = e^{-x v_{s,t}(\theta)}, \theta >0$, and its differentiated form, with the help of the monotone convergence theorem we get
\begin{equation}\label{eq:derivative_BP}
\mathbb P^{(s,x)}[Z_t<\infty] = e^{-x v_{s,t}(0)} \qquad \text{and}\qquad \mathbb E^{(s,x)}[Z_t \ind_{\{Z_t <\infty\}}] = xv_{s,t}'(0)e^{-x v_{s,t}(0)}.
\end{equation}
It is then straightforward to see that \ref{first_moment1} is equivalent to~\ref{first_moment1.5} and that   \ref{first_moment0} is equivalent to~\ref{first_moment2}.

Obviously,~\ref{first_moment2}~$\Rightarrow$~\ref{first_moment1.5}. To prove the converse implication, suppose that~\ref{first_moment1.5} holds. If $v_{s,t}(0)>0$ for some $(s,t)\in\Delta$, then passing to the limit as ${\theta \to0^+}$ in the identity $v_{0,t}(\theta)={v_{0,s}\big(v_{s,t}(\theta)\big)}$ would yield $v_{0,t}(0)=v_{0,s}\big(v_{s,t}(0)\big)\neq0$ because $v_{0,s}$ is a self-map of~$\UH$. Hence $v_{s,t}(0)=0$ for all ${(s,t)\in\Delta}$. Note that ${v_{0,s}'(0)>0}$ for any $s\ge0$ because\footnote{Alternatively, the inequality ${v_{0,s}'(0)>0}$ follows from the fact that the function $v_{0,s}$ is positive and concave on ${(0,\infty)}$ and vanishes at~$0$.} the angular derivative of a holomorphic self-map at a boundary fixed point ${\sigma\neq\infty}$ cannot vanish, see e.g.\  \cite[Proposition~1.9.3 on p.\,54]{BCD-Book}.  Taking the derivative in both sides and then passing to the limit we may, therefore, conclude that $v_{s,t}'(0) = {v_{0,t}'(0)}/{v_{0,s}'(0)}<\infty$.

It remains to show that \ref{first_moment2}~$\Leftrightarrow$~\ref{first_moment3}~$\Leftrightarrow$~\ref{first_moment4}. To begin with, taking into account the relationship between evolution families and \emph{reverse} evolution families, see e.g.\  \cite[Remark~2.8]{GHP},  the equivalence \ref{first_moment2}~$\Leftrightarrow$~\ref{first_moment3} follows from \cite[Theorem 1.1]{BRFP2015}.
Furthermore,
\begin{equation*}
\phi'(\theta,t) = a_t + 2 b_t \theta + \int_{(0,1)} x(1-e^{-\theta x})\,\pi_t(\di x) -  \int_{[1,\infty)} xe^{-\theta x}\,\pi_t(\di x).
\end{equation*}
Hence, passing to the limit $\theta\to0^+$ yields that
\begin{equation*}
\phi'(0,t) = a_t - \int_{[1,\infty)} x \,\pi_t(\di x).
\end{equation*}
Because $t\mapsto a_t$ is  locally integrable by the definition of a L\'evy family and because ${q_t=-\phi(0,t)}$, the equivalence between \ref{first_moment3} and \ref{first_moment4} now follows.

Finally, the formula for the spectral function in \cite[Theorem 1.1]{BRFP2015} implies that
\begin{equation*}
v_{s,t}'(0) = \exp\left[ -\int_s^t \phi'(0,r)\,\di r\right].\qedhere
\end{equation*}
\end{proof}

Let us now explore the second moments of~$(Z_t)$.
Consider  a Herglotz vector field~$\phi$ whose reverse evolution family $(v_{s,t})$ is contained in~$\BF$. Then by \cite[Corollary~3.3]{GHP}, $\phi(\cdot,t)\in\Gen(\BF)$ for a.e.\,${t\ge0}$. Hence, see e.g. \cite[Theorem~3]{GHP}, $\phi''(\cdot,t)$ is completely monotone on~$(0,\infty)$, from which we get
$
\phi'(1,t)-\phi''(0,t)\le\phi'(0,t)\le\phi'(1,t)
$
 for a.e. ${t\ge0}$.
It is evident from the proof of Proposition~\ref{prop:int} that $t\mapsto \phi'(1,t)$ is in~$\Lloc$. Therefore, if ${t\mapsto \phi''(0,t)}$ is in~$\Lloc$, so is ${t\mapsto \phi'(0,t)}$. With this observation and \cite[Proposition~3.17]{GHP} taken into account, a reasoning similar to that in the proof of Theorem~\ref{thm:first_moment}  yields  the following characterization of finite second moments.

\begin{theorem}\label{thm:second_moment}
Under the hypothesis of Theorem~\ref{thm:first_moment},
the following six conditions are equivalent:
\begin{enumerate}[label=\rm(\arabic*${}'$)]

\item\label{second_moment0} $\mathbb E^{(s,x)}[Z_t^2]<\infty$ for all $x \in (0,\infty)$ and $(s,t)\in \Delta$.

\item\label{second_moment1} $\mathbb E^{(0,1)}[Z_t^2]<\infty$ for all $t \ge 0$.

\item \label{second_moment1.5} $v_{0,t}(0)=0$ and $v''_{0,t}(0)>- \infty$ for all $t\ge 0$.

\item \label{second_moment2} $v_{s,t}(0)=0$ and $v''_{s,t}(0)>-\infty$ for all $(s,t)\in\Delta$.

\item\label{second_moment3} $\phi(0,t)=0$ for a.e.\ $t\ge0$ and the function $t\mapsto \phi''(0,t)$ is in $\Lloc$.

\item \label{second_moment4} $q_t=0$ for a.e.\ $t\ge0$ and the function $t\mapsto \int_{[1,\infty)} x^2\,\pi_t(\di x)$ is in $\Lloc$.

\end{enumerate}
If one (and hence all) of the above conditions hold, then $v_{s,t}'(0)$ is finite  and
\begin{equation}\label{EQ_second-moment}
\mathbb E^{(s,x)}[Z_t^2] = \big(x v_{s,t}'(0)\big)^2 - x v_{s,t}''(0) = x v_{s,t}'(0) \left[ x v_{s,t}'(0) +  \int_s^t \phi''(0,r)v_{r,t}'(0) \,\di r \right].
\end{equation}
 for all $x \in (0,\infty)$ and $(s,t) \in \Delta$.
\end{theorem}

\subsection{Processes with absolutely continuous mean and variance}
According to Theorem~\ref{thm:REF_BP}, the Laplace exponents of a branching process~$(Z_t)$ form a topological reverse evolution family~$(v_{s,t})$ contained in $\BF$, and vice versa. In order to get the ODE in Theorem~\ref{thm:main_BP}, one also needs the absolute continuity, i.e.\ property~\ref{REF3}.

Various conditions assuring that a given topological (reverse) evolution family is actually absolutely continuous are known, see e.g.\  \cite[Section~3.4]{GHP}. Applying one of them to the Laplace exponents of~$(Z_t)$ allows us to  establish a \textit{one-to-one correspondence} between branching processes with sufficiently regular behaviour of the first two moments and a class of L\'evy families. Namely, we prove the following result.

\begin{theorem}\label{TH_1-to-1-ACloc}
Let $(Z_t)_{t\ge0}$ be a branching process and  $(v_{s,t})\subset\BF$ the topological reverse evolution family formed by the Laplace exponents of~$(Z_t)$.
The following conditions are equivalent:
\begin{enumerate}
\item[(i)] for any ${s\ge0}$ and any ${x\ge0}$, the two first moments ${\mathbb E^{(s,x)}[Z_t^k]}$, ${k=1,2}$, are (finitely-valued) locally absolutely continuous functions of~$t$ on~${[s,\infty)}$;

\item[(ii)] the reverse evolution family $(v_{s,t})$ is absolutely continuous and the Herglotz vector field~$\phi$ associated with~$(v_{s,t})$ satisfies the equivalent conditions \ref{second_moment3}\,--\,\ref{second_moment4} in Theorem~\ref{thm:second_moment}.
\end{enumerate}
\end{theorem}
\begin{proof}
The implication~(ii)\,$\Rightarrow$\,(i) holds in view of Theorems~\ref{thm:first_moment} and~\ref{thm:second_moment}. To prove the converse implication, suppose that~$(Z_t)$ satisfies condition~(i). Then  ${\mathbb E^{(0,1)}[Z_t]<\infty}$ and as a consequence,  ${v_{0,t}(0)=0}$ for all~$t\ge0$. Moreover,
\[
 \mathbb E^{(0,1)}[Z_t] = v_{0,t}'(0) \quad \text{and} \quad \mathbb E^{(0,1)}[Z_t^2] = [v_{0,t}'(0)]^2 - v_{0,t}''(0).
\]
Therefore, by~\cite[Corollary~3.10]{GHP}, the Laplace exponents of $(Z_t)$ form an absolutely continuous reverse evolution family~$(v_{s,t})$ contained in~$\BF$. According to Theorem~\ref{thm:main_BP}, the Herglotz vector field~$\phi$ associated with~$(v_{s,t})$ admits the Silverstein-type representation~\eqref{eq:BF_generator}, and it only remains to apply Theorem~\ref{thm:second_moment} to see that $\phi$ satisfies conditions \ref{second_moment3}\,--\,\ref{second_moment4}.
\end{proof}

\subsection{Extinction, explosion, and the Denjoy\,--\,Wolff point}
A fundamental role for understanding dynamics of a holomorphic self-map is played by the so-called \textsl{Denjoy\,--\,Wolff point} (\hbox{often} abbreviated as \textsl{DW-point}). This point is the common limit of the orbits under iteration of the self-map; see, e.g.,  Appendix~\ref{App-Fixed} or~\cite[Sect.\,2.3]{GHP} for the precise definition and  basic details. Below we explore how the position of the DW-point of the Laplace exponents is related to the probabilistic properties of a branching process.

Denote by $\BF_\tau$ the class of Bernstein functions consisting of~$\id_\UH$ and all $v\in\BF\setminus\{\id_\UH\}$ having the DW-point at~$\tau$.
The Laplace exponents $v_t$ of a time-homogeneous branching process form a one-parameter semigroup and hence $(v_t)_{t\ge0}\subset\BF_\tau$ for some~${\tau\in[0,\infty]}$, see Remarks~\ref{RM_the-same-FPs}  and~\ref{RM_DW-Bernstein}.

In the time-inhomogeneous case, the Laplace exponents form a reverse evolution family $(v_{s,t})$ and, in general, the position of the DW-point of~$v_{s,t}$ may depend on the time parameters $s$ and~$t$. However, it seems to be interesting to investigate (and give a probabilistic interpretation of) the special case~$(v_{s,t})\subset\BF_\tau$. First we will consider the case of the Denjoy\,--\,Wolff point on the boundary, which means ${\tau\in\{0,\infty\}}$. The case of the Denjoy\,--\,Wolff point ${\tau\in(0,\infty)}$ is covered in Subsect.\,\ref{subsub:infty}. It is worth mentioning that in certain questions, extending consideration to boundary regular fixed points\footnote{In what follows abbreviated as \textsl{BRFP}; see Appendix~\ref{App-Fixed} or \cite[Sect.\,2.3]{GHP}  for the definition.} at~$0$ or~$\infty$ is rather natural or, at least, does not make the picture more complicated. That is why several results in this section are established in this more general case.

\addtocontents{toc}{\SkipTocEntry}\subsubsection{The DW-point at the origin}\label{subsub:0}
A Bernstein function ${v\neq\id_\UH}$ has the DW-point at~0 if and only if $v(0)=0$ and $v'(0)\le 1$. Therefore, the Laplace exponents of a branching process $(Z_t)$ are contained in~$\BF_0$ if and only if the process is conservative and ``subcritical'', i.e.\ ${\mathbb E^{(s,x)}[Z_t]\le x}$ for all ${x \in (0,\infty)}$ and $(s,t)\in \Delta$; see the proof of Theorem~\ref{thm:first_moment}.

Assuming that the Laplace exponents of a branching process $(Z_t)$ belong to~$\BF_0$ and that its sample paths are of class ${D[0,\infty)}$, see Theorem~\ref{thm:cadlag}, we are able to give a sufficient condition for almost sure extinction within a finite time.

Recall that by definition, the \textsl{extinction time} is
\begin{equation}\label{EQ_ext-time}
{T_0^s \coloneqq \inf\{t \ge s: Z_t =0\}}.
\end{equation}
Recall also that there is a natural one-to-one correspondence between absolutely continuous reverse evolution families and Herglotz vector fields; see Theorem~\ref{TH_mainBCM1} and Theorem~\ref{thm:main_BP}.
\begin{theorem}\label{TH_infty-to-zero}
Let $(v_{s,t})$ be an absolutely continuous reverse evolution family in~$\BF_0$ with associated Herglotz vector field~$\phi$. Suppose that
\begin{equation}\label{EQ_int-phi-two-primes}
\int_0^{+\infty}\phi''(\infty,t)\,\di t=+\infty.
\end{equation}
Then $\lim\nolimits_{t\to\infty} v_{s,t}(\infty)=0$  for any ${s\ge0}$, and hence for the branching process~$(Z_t)$ with $D[0,\infty)$-paths associated to~$(v_{s,t})$, we have
$$
 \qquad\qquad\mathbb P^{(s,x)}[T_0^s<\infty]=1\quad\text{~for any $~x>0~$ and any~$~s\ge0$.}
$$
\end{theorem}
Before proving this theorem, we will first characterize Herglotz vector fields of reverse evolution families in~$\BF_0$ with the help of an integral representation for $\Gen(\BF_0)$, which can be found e.g.\ in \cite[Corollary~3.6]{GHP} and \cite[Theorem~1 on p.\,21]{LeGall}{\hspace{.05em}}\footnote{Worth recalling that from the probabilistic point of view, elements of $\Gen(\BF_0)$ are conservative branching mechanisms that are subcritical in the sense explained above.}. Namely, a function ${\phi:\UH\to\Complex}$ belongs to~$\Gen(\BF_0)$ if and only if there exist $c,\,b\ge0$ and a non-negative Borel measure~$\pi$ on~$(0,\infty)$ with $\int_{(0,\infty)}\min\{\lambda^2,\lambda\}\,\pi(\di\lambda)<\infty$ such that
\begin{equation}\label{EQ_inf-gen-represent-DWzero}
\phi(\zeta)= c\zeta + b\zeta^2 + \int_{(0,\infty)} \!\!\big(e^{-\zeta x} - 1 + \zeta x\big) \,\pi(\di x) \quad  \text{for all~}~\zeta\in \UH.
\end{equation}
\vspace{-3ex}
\begin{remark}\label{RM_comparing two-representations-for-DWzero}
Clearly, any $\phi\in\Gen(\BF_0)$ admits also Silverstein's representation~\eqref{gen}. As we mentioned in \cite[right after Corollary~3.6]{GHP}, the coefficient $b$ and the measure~$\pi$ in the representations~\eqref{gen} and~\eqref{EQ_inf-gen-represent-DWzero} are the same. Moreover, it is easy to show that
\begin{align*}
-q&=\phi(0):=\lim_{\theta\to0^+}\phi(\theta)=0, & 2b&=\phi''(\infty):=\lim_{\theta\to+\infty}\phi''(\theta),\\
c&=\phi'(0):=\lim_{\theta\to0^+}\phi'(\theta), & a&=c+\textstyle\int_{[1,\infty)}\lambda\,\pi(\di\lambda).
\end{align*}
\end{remark}
\begin{theorem}\label{thm:BP-for-DWzero}
A function $\phi:\UH\times[0,+\infty)\to\Complex$ is a Herglotz vector field of a reverse evolution family contained in~$\BF_0$ if and only if for a.e.~$t\ge0$ it admits the representation
\begin{equation}\label{eq:BF_generator-DWzero}
\phi(\zeta,t) =c_t \zeta + b_t \zeta^2 + \int_0^\infty (e^{-\zeta x} - 1 + \zeta x) \,\pi_t(\di x), \qquad \zeta \in \UH,
\end{equation}
with some $c_t,\,b_t\ge0$, and some non-negative Borel measures $\pi_t$ on~$(0,\infty)$ satisfying the following two conditions:
\begin{itemize}
 \item[(a)] $t\mapsto c_t$ and $t\mapsto b_t$ are in $\Lloc$;
 \item[(b)] the family of measures $\min\{\lambda^2,\lambda\}\,\pi_t(\di\lambda)$ on $(0,\infty)$ is locally integrable\footnote{Recall  that a family of measures $(\rho_t)_{t\ge0}$ on~$[0,\infty)$ is called locally integrable if $t\mapsto\rho_t(B)$ is $L^1_{\mathrm{loc}}$ for any Borel set ${B\subset[0,\infty)}$.} in~$t$.
\end{itemize}
The family $(c_t,b_t,\pi_t)$ is determined uniquely up to a Lebesgue null-set on the $t$-axis.
\end{theorem}

\begin{proof}
Recall that $\BF$ is topologically closed in~$\Hol(\UH,\UH)$. Moreover, it is well-known, see e.g.\  \cite[Remark~2.4]{GumProkh2018}, that if $v\in\Hol(\UH,\UH)\setminus\{\id_\UH\}$ is the limit of a sequence of self-maps with the DW-point at~$0$, then the DW-point of~$v$ is also at~$0$. It follows that $\BF_0$ is a topologically closed semigroup in~$\Hol(\UH,\UH)$. Therefore, according to \cite[Theorem~2]{GHP}
the absolutely continuous reverse evolution family generated by a Herglotz vector field $\phi$ is contained in~$\BF_0$ if and only if ${\phi(\cdot,t)\in\Gen(\BF_0)}$ for a.e.\,$t\ge0$.

Thus, taking into account representation~\eqref{EQ_inf-gen-represent-DWzero}, we see that it is sufficient to prove  the following statement: let ${\phi:\UH\times[0,\infty)\to\Complex}$ be defined by~\eqref{eq:BF_generator-DWzero} with some $c_t,\,b_t\ge0$ and some non-negative Borel measures $\pi_t$ on $(0,\infty)$; then $\phi$ is Herglotz vector field if and only if conditions~(a) and~(b) are satisfied.

Suppose $\phi$ is a Herglotz vector field. Then by
Remark~\ref{RM_comparing two-representations-for-DWzero} and Proposition~\ref{prop:int}, we have:
\begin{itemize}
\item[(i)] ${t\mapsto b_t}$ is in $\Lloc$;
\item[(ii)] $t\mapsto c_t+\int_{[1,\infty)}\lambda\,\pi_t(\di\lambda)$ is in $\Lloc$;
\item[(iii)] the family of measures $\min\{\lambda^2,1\}\,\pi_t(\di\lambda)$ on $(0,\infty)$ is locally integrable in $t$.
\end{itemize}
In view of Lemma~\ref{lem:measurable}, assertion~(iii) implies that the family of measures $\min\{\lambda^2,\lambda\}\,\pi_t(\di\lambda)$ is a Borel kernel and hence, the function $I_2(t):=\int_{[1,\infty)}\lambda\,\pi_t(\di\lambda)$ is measurable. Therefore, taking into account that ${c_t\ge0}$ and using~(ii), we conclude that both $t\mapsto c_t$ and~$I_2$ are in $\Lloc$. Finally, by (iii), $I_1(t):=\int_{(0,1)}\lambda^2\,\pi_t(\di\lambda)$ is also in~$\Lloc$.
Now~(a) and~(b) follow immediately.

Conversely, suppose that $c_t,\,b_t\ge0$ and  non-negative Borel measures $\pi_t$ on $(0,\infty)$ satisfy conditions~(a) and~(b). Then using Remark~\ref{RM_comparing two-representations-for-DWzero} and Lemma~\ref{lem:measurable}, it is easy to show that $\phi$ defined by~\eqref{eq:BF_generator-DWzero} admits also representation~\eqref{eq:BF_generator} with a certain L\'evy family $(q_t,a_t,b_t,\pi_t)$. By Proposition~\ref{prop:int}, this means that~$\phi$ is a Herglotz vector field.
\end{proof}

\begin{remark}\label{RM-charac-BRFP-at-zero}
Note that the representation~\eqref{EQ_inf-gen-represent-DWzero} with $c\in\Real$ of arbitrary sign is commonly used in the time-homogeneous case. It characterises branching mechanisms~$\phi$ of stochastic processes with finite mean or, equivalently, infinitesimal generators of one-parameter semigroups in~$\BF$ with a BRFP at~$0$. A complex-analytic proof of this fact can be found in~\cite[Corollary~3.6]{GHP}. In the time-inhomogeneous case, one can use Theorem~\ref{thm:first_moment} to establish an analog of the theorem proved above, which provides a representation formula for Herglotz vector fields corresponding to branching processes with finite mean. It is literally the same representation~\eqref{eq:BF_generator-DWzero} as in Theorem~\ref{thm:BP-for-DWzero}, with $c_t$, $b_t$, and $\pi_t$ satisfying (a) and~(b), but with the condition~$c_t\ge0$ relaxed to~$c_t\in\Real$ for all~${t\ge0}$. A subtle difference of this more general setting is illustrated by the example given below: there exist Herglotz vector fields~$\phi$ such that representation~\eqref{eq:BF_generator-DWzero} holds but the parameters do not satisfy the integrability conditions (a) and~(b). Observe that according to the proof of Theorem~\ref{thm:BP-for-DWzero}, such a situation cannot happen when~$c_t\ge0$ for a.e.~$t\ge0$.
\end{remark}
\begin{example}\label{EX_counterexample}
Let $\lambda(t):=1/t$ for $t\in(0,1)$ and $\lambda(t):=1$ for $t\in\{0\}\cup[1,\infty)$. Define
$$
\phi(\zeta,t):=\exp(-\zeta\lambda(t))-1,\quad t\ge0,~\zeta\in\UH.
$$
Then $\phi$ admits representation~\eqref{eq:BF_generator} with $q_t=a_t=b_t=0$ and $\pi_t:=\delta_{\lambda(t)}$. Therefore, by Proposition~\ref{prop:int}, $\phi$ is a Herglotz vector field. The reverse evolution family $(v_{s,t})$ associated with~$\phi$ is contained in~$\BF$ by Theorem~\ref{thm:main_BP}. Moreover, $\phi$ satisfies  equality~\eqref{eq:BF_generator-DWzero} with the same $\pi_t=\delta_{\lambda(t)}$, ${b_t=0}$, and with ${c_t:=-\lambda(t)}$ for all ${t\ge0}$. Clearly, $t\mapsto c_t$ is not $\Lloc$. As a result, $0$ is not a BRFP of $v_{0,t}$ for any~$t>0$.
\end{example}

As the last preparatory step before proving Theorem~\ref{TH_infty-to-zero}, observe that by equality~\eqref{eq:laplace} relating $v_{s,t}$'s to transition probabilities,
$
 \mathbb P^{(s,x)}[Z_t =0] = e^{-x v_{s,t}(\infty)}
$
for all $(s,t)\in\Delta$ and all $x>0$.
Under the hypothesis that $(Z_t)$ has $D[0,\infty)$-paths, we have $\{T_0^s<\infty\} = \bigcup_{n\in \N,\, n>s} \{Z_n=0\}$ because the event $\{Z_t=0\}$ is increasing w.r.t.\ the parameter~$t$. As a result,  by the continuity of measures,
\begin{equation}\label{EQ_finite-ext-time-probab}
\mathbb P^{(s,x)}[T_0^s <\infty] = \lim_{t\to\infty} e^{-x v_{s,t}(\infty)}.
\end{equation}

\begin{proof}[\proofof{Theorem~\ref{TH_infty-to-zero}}]
According to Theorem~\ref{thm:BP-for-DWzero}, the Herglotz vector field~$\phi$  is represented for a.e.\,${t\ge0}$ by~\eqref{eq:BF_generator-DWzero}
with $c_t,\,b_t\ge0$ and some non-negative Borel measures $\pi_t$ on~$(0,\infty)$ satisfying the integrability conditions~(a) and~(b).
According to Remark~\ref{RM_comparing two-representations-for-DWzero}, ${\phi''(\infty,t) = 2b_t \ge0}$ for a.e. ${t\ge0}$. Therefore, the integral in~\eqref{EQ_int-phi-two-primes} is well-defined.

Note that $e^{-\theta x} - 1 +\theta x \ge0$ for any $\theta, x >0$. Hence,
\begin{equation}\label{EQ_phi-phi0}
\phi(\theta, t) \ge \phi_0(\theta,t):= b_t \theta^2 \qquad\text{for a.e.\ $t\ge0$ and all~$\theta>0$}.
\end{equation}
It is easy to see that $\phi_0$ is a Herglotz vector field. Denote by $(v^0_{s,t})$ the reverse evolution family associated to $\phi_0$.
Using the ODE~\eqref{EQ_ini-LK-ODE}, it is then easy to check that for all $z\in\UH$ and all~$(s,t)\in\Delta$, $$v^0_{s,t}(z)=\frac{z}{1+b(s,t)z},\quad b(s,t):=\frac12\int_s^t\phi''(\infty,r)\,\di r.$$
Taking into account~\eqref{EQ_finite-ext-time-probab}, in order to prove the theorem it is enough to show that
\begin{equation}\label{EQ_v-v0}
v_{s,t}(x)\le v^0_{s,t}(x)\quad\text{for all $x>0$ and all $(s,t)\in\Delta$}.
\end{equation}

Fix arbitrary $T>0$ and $x>0$ and denote
$$
 f(t):=v_{T-t,T}(x),\quad g(t):=v^0_{T-t,T}(x),\quad \beta := \max \{f(t), g(t): t\in [0,T]\}.
$$
Then $f(0)=x=g(0)$. Moreover, by Theorem~\ref{TH_mainBCM1} and by~\eqref{EQ_v-v0}, for a.e.~$t\in[0,T]$ we have
$$
f'(t)= - \phi\big(f(t), T-t\big)\le - \phi_0\big(f(t), T-t\big)\quad\text{and}\quad g'(t)=-\phi_0\big(g(t),T-t\big).
$$
Since $\big|\phi_0(y,T-t) - \phi_0(z,T-t)\big| \le 2\beta b_t|y-z|$ for a.e.~${t\in[0,T]}$ and all ${y,\,z \in [0,\beta]}$, it follows with the help of a standard argument (see e.g. Lemma~\ref{lem:comparison} in the Appendices) that $f\le g$ on~$[0,T]$. This proves~\eqref{EQ_v-v0}, as required.
\end{proof}

\addtocontents{toc}{\SkipTocEntry}\subsubsection{The Denjoy\,--\,Wolff point at infinity} \label{subsub:infty}
Recall that $f\in\Hol(\UD,\UD)\setminus\{\id_\UH\}$ has the DW-point at~$\infty$ if and only if $\anglim_{z\to\infty} f(z)/z\ge1$. If $f\big((0,\infty)\big)\subset(0,\infty)$, then using Wolff's Lemma, see e.g.\  \cite[Theorem~B]{GHP}, we see that $\infty$ is the DW-point of~$f$ if and only if for some and hence for any~$\theta_0\ge0$ the inequality $f(\theta)\ge\theta$ holds for all ${\theta>\theta_0}$.

We start by establishing a probabilistic interpretation for the DW-point at~$\infty$. Next we characterize Herglotz vector fields that generate reverse evolution families in~$\BF$ with the DW-point (or more generally, BRFP) at~$\infty$ and use this characterization to obtain a sufficient condition for almost sure explosion within a finite time.

\begin{theorem}\label{thm:non-decreasing}
Let $(Z_t)$ be a branching process and let $(v_{s,t})$ be the associated topological reverse evolution family. The following four conditions are equivalent:
\begin{enumerate}[label=\rm(\alph*)]
\item\label{item:a} $(v_{s,t})\subset\BF_\infty$;
\item\label{item:a2} for any $0 \le s\le t \le u$ and any $x>0$, $\mathbb P^{(s,x)}\big[Z_{t}\le Z_u\big]=1$;
\item\label{item:b} for any $(s,t)\in\Delta$ and any $x>0$, $\mathbb P^{(s,x)}\big[Z_{t}\ge x\big]=1$;
\item\label{item:c} for any $(s,t)\in\Delta$ there exists $x>0$  such that\, $\mathbb P^{(s,x)}\big[Z_{t}\ge x\big]=1$.
\end{enumerate}
Moreover, if $(Z_t)_{t\ge0}$ has \cadlag paths, then the above conditions are further equivalent to that $(Z_t)$ has non-decreasing paths almost surely.
\end{theorem}
\begin{proof}
Note that in case $(Z_t)$ has \cadlag paths, it is not difficult to see that~\ref{item:a2} is equivalent to having $${\mathbb P^{(s,x)}\big[Z_t \le Z_u\text{~whenever~$s\le t\le u$}\big]\,=\,1}$$ for any ${x>0}$ and any ${s\ge0}$.

To show the equivalence among the four conditions, we first suppose that~\ref{item:a} holds. As mentioned above, condition \ref{item:a} is equivalent to that $v_{s,t}(\theta)\ge\theta$ for all $\theta>0$ and $(s,t) \in \Delta$.
Let $0<u<x$ and $\theta>0$. Then
\begin{equation*}
e^{-x v_{s,t}(\theta)}= \int_{[0,\infty)} e^{-  \theta y} k_{s,t}(x, \di y) \ge  \int_{[0,u]} e^{-  \theta y} k_{s,t}(x, \di y) \ge   e^{-  \theta u} k_{s,t}(x,[0,u])
\end{equation*}
\vspace{.75ex}
and hence \qquad\quad
$
k_{s,t}(x, [0,u])  \le e^{\theta u}e^{-x v_{s,t}(\theta)} \le e^{ \theta u}e^{-x\theta}=e^{-\theta (x-u)}.
$\\[.75ex]
Letting $\theta\to +\infty$  yields $k_{s,t}(x,[0,u])=0$ for all $u\in(0,x)$. It follows that
$ k_{s,t}(x,[0,x))=0,$ or equivalently that $\mathbb P^{(s,x)}\big[Z_{t}<x\big]=0.$ This proves the implication \ref{item:a}~$\Rightarrow$~\ref{item:b}.

Obviously, \ref{item:b}~$\Rightarrow$~\ref{item:c}.
Now suppose that \ref{item:c} holds, or equivalently that $ k_{s,t}(x,[0,x))=0$ for any $(s,t)\in\Delta$ and some~$x>0$, which may depend on~$s$ and~$t$.
Then
\begin{equation*}
e^{-x v_{s,t}(\theta)}= \int_{[x,\infty)} e^{-  \theta y} k_{s,t}(x, \di y) \le  \int_{[x,\infty)} e^{-  \theta x} k_{s,t}(x, \di y) \le e^{-\theta x}.
\end{equation*}
 Hence, we get $v_{s,t}(\theta) \ge \theta$ for any ${(s,t)\in\Delta}$ and any ${\theta>0}$, i.e.\ ${(v_{s,t})\subset\BF_\infty}$.

\newcommand{\Delt}{\Delta^{\!\infty}}

It remains to establish the equivalence of \ref{item:a2} and \ref{item:b}. The implication  \ref{item:a2}~$\Rightarrow$~\ref{item:b} is obvious. Conversely, suppose that \ref{item:b} holds, which can be written as $k_{t,u}\big(y, [y,\infty]\big)=1$ for all ${(t,u) \in \Delta}$ and all ${y>0}$. Denote $\Delt:= \{(y,z): 0 \le y \le z \le \infty\}$.
Using \eqref{eq:FD}, for any ${n \in \N_0}$ and any ${0 \le s \le t \le u <\infty}$ we get
\begin{align*}
\mathbb P^{(s,x)}[ Z_t \le Z_u]
&=\mathbb P^{(s,x)}[(Z_t, Z_u)\in\Delt] = \int_{\Delt}\!\! k_{s,t}(x, \di y)\,k_{t,u}(y,\di z)  \\
&= \int_{[0,\infty]}\! k_{s,t}(x, \di y) ~=~1,\mathrlap{\quad\text{as desired.}} \qedhere
\end{align*}
\end{proof}

We are able to give a probabilistic interpretation also for the angular derivative at~$\infty$, which for any  $v\in\BF$ can be calculated as $v'(\infty)=\lim_{(0,\infty)\ni\theta\to\infty}v(\theta)/\theta$; see Appendix~\ref{App-Fixed}.
\begin{remark}
Recall that $\mathbb E^{(s,x)}[e^{-\theta Z_t}] = e^{-x v_{s,t}(\theta)}$. This can be written as
$$
\|e^{-Z_t} \|_{L^\theta(\mathbb P^{(s,x)})}= \exp\big(-x {v_{s,t}(\theta)}/{\theta}\big).
$$
 Hence, taking the limit $\theta\to\infty$, we get
$
\|e^{-Z_t} \|_{L^\infty(\mathbb P^{(s,x)})} = e^{-x v_{s,t}'(\infty)}.
$
This is equivalent to
$
\text{$\mathbb P^{(s,x)}$-ess}\inf Z_t = x v_{s,t}'(\infty).
$
In particular, $\infty$ is a BRFP of~$v_{s,t}$ if and only if $\mathbb P^{(s,x)}\text{-ess}\inf Z_t>0$ for some and hence all $x>0$.   Similarly, $\infty$ is a DW-point of~$v_{s,t}$ if and only if ${\mathbb P^{(s,x)}\text{-ess}\inf Z_t\ge x}$ for some and hence all ${x>0}$.  The latter is yet another way to express ``monotonicity'' of the process $(Z_t)$, i.e.\ another condition equivalent to conditions \ref{item:b}\,--\,\ref{item:c} in Theorem~\ref{thm:non-decreasing}.
\end{remark}
\begin{theorem}\label{TH_HF-DW-infty}
A function $\phi:\UH\times[0,\infty)\to\Complex$ is a Herglotz vector field of a reverse evolution family $(v_{s,t}) \subset \BF$ with a common BRFP at $\infty$ if and only if it admits the following representation:
\begin{equation}\label{eq:BRFP_infty}
\phi(\zeta, t) = -q_t+  d_t \zeta + \int_0^\infty (e^{-\zeta x} - 1) \,\pi_t(\di x), \qquad \zeta \in \UH, ~ a.e.\ t\ge0,
\end{equation}
with some $q_t\ge0$, $d_t \in \R$ and some non-negative measures $\pi_t$  on $(0,\infty)$ satisfying the following conditions:
\begin{itemize}
\item[\rm (a)]  $t\mapsto q_t$ and $t\mapsto d_t$ are in $\Lloc$;
\item[\rm (b)] $\min \{x,1\} \pi_t(\di x)$ is locally integrable in $t\ge0$.
\end{itemize}
In such a case,
\begin{equation}\label{EQ_der-at-infty}
v_{s,t}'(\infty)=\exp\Big(-\int_{s}^{t}d_r\,\di r\Big)\quad\text{for all~$(s,t)\in\Delta$.}
\end{equation}
\end{theorem}
\begin{remark}\label{RM_HF-DW-infty}
If $\phi$ admits representation~\eqref{eq:BRFP_infty} with the parameters satisfying the conditions stated in the above Theorem~\ref{TH_HF-DW-infty}, then $\phi$ admits also the Silverstein-type representation~\eqref{eq:BF_generator} with $a_t:=d_t-\int_{(0,1)}x\,\pi_t(\di x)$, ${b_t:=0}$ and with the same $q_t$ and~$\pi_t$ as in~\eqref{eq:BRFP_infty}.
\end{remark}
Before proving Theorem~\ref{TH_HF-DW-infty}, we establish  two direct corollaries.
Recall that $v\in\BF$ belongs to~$\BF_\infty$ if and only if~$v'(\infty)\ge0$. Therefore, with~\eqref{EQ_der-at-infty} taken into account, Theorem~\ref{TH_HF-DW-infty} implies the following characterization of Herglotz vector fields of absolutely continuous reverse evolution families contained in~$\BF_\infty$.
\begin{corollary}\label{cor:DW_infty}
 A function $\phi:\UH\times[0,\infty)\to\Complex$ is a Herglotz vector field of a reverse evolution family $(v_{s,t}) \subset \BF_\infty$  if and only if $\phi$ satisfies the necessary and sufficient condition of Theorem~\ref{TH_HF-DW-infty} with~${d_t \le0}$ for a.e.~${t\ge0}$.
\end{corollary}

As the following example shows, an observation analogous to that in Remark~\ref{RM-charac-BRFP-at-zero} is valid for the fixed point at~$\infty$: if a Herglotz vector field~$\phi$ admits representation~\eqref{eq:BRFP_infty} with ${d_t\le0}$ for a.e.~${t\ge0}$, then the parameters automatically satisfy the integrability conditions stated in Theorem~\ref{TH_HF-DW-infty}; this is not the case any more without the assumption that ${d_t\le0}$.
\begin{example}\label{EX_example}
  Let \textcolor{black}{${\phi(\zeta,0):=0}$ and $\phi(\zeta,t):=t^{-1}\zeta+t^{-2}(e^{-t\zeta}-1)$ for all ${t>0}$ and all ${\zeta\in\UH}$. This is a Herglotz vector field by the equivalence of (i) and~(ii) in Proposition~\ref{prop:int}. Clearly, $\phi$~is of the form~\eqref{eq:BRFP_infty} for each~$t\ge0$, but $d_t=\phi'(\infty,t)=t^{-1}$ is not~$\Lloc$.}
\end{example}

In view of~\eqref{EQ_finite-ext-time-probab}, we have another immediate corollary of Theorem~\ref{TH_HF-DW-infty}, which gives a sufficient condition for almost surely no extinction within a finite time. With the help of  Theorem~\ref{thm:cadlag} we may assume that the branching process~$(Z_t)$ associated to a given reverse evolution family ${(v_{s,t})\subset\BF}$ has sample paths of class~$D[0,\infty)$.
\begin{corollary}\label{TH_infty-to-nonzero}
Suppose that $\phi$ satisfies the necessary and sufficient condition of Theorem~\ref{TH_HF-DW-infty}. Then for the associated branching process~$(Z_t)$ with $D[0,\infty)$-paths, we have $\mathbb P^{(s,x)}[T_0^s <\infty]=0$ for any ${x>0}$ and any~${s\ge0}$.
\end{corollary}
\begin{proof}[\proofof{Theorem~\ref{TH_HF-DW-infty}}]
Suppose that $\phi:\UH\times[0,\infty)\to\Complex$ is a Herglotz vector field. Denote by $(v_{s,t})$ the reverse evolution family associated with~$\phi$.
Combining \cite[Theorem~D, Corollary~3.3, and Corollary~3.7]{GHP},  one can conclude that $(v_{s,t})$ is contained in~$\BF$ and has a common BRFP at $\infty$ if and only if the following two conditions are satisfied:
\begin{itemize}
\item[(i)] for a.e.\ $t\ge0$, $\phi(\cdot, t)$ is of the form~\eqref{eq:BRFP_infty} for some $q_t \ge0, d_t \in \R$ and non-negative measures $\pi_t$ such that $\int_0^\infty \min\{x,1\} \pi_t(\di x)<\infty$;

\item[(ii)] $t\mapsto \phi'(\infty, t)=d_t$ is in $\Lloc$.
\end{itemize}
In view of Remark~\ref{RM_HF-DW-infty} and Proposition~\ref{prop:int}, condition~(i) implies that $t\mapsto q_t$ is in $\Lloc$ and that $(\pi_t)_{t\ge0}$ is a Borel kernel. It further follows that conditions (i) and~(ii) imply that $\min \{x,1\} \pi_t(\di x)$ is locally integrable in~$t\ge0$, because
\[
 \int_0^\infty (1-e^{-x} ) \,\pi_t(\di x)  =  -q_t +d_t -\phi(1,t)
\]
and $\min\{x,1\}\le e(e-1)^{-1}(1-e^{-x})$ on $(0,\infty)$. This proves the necessity part.

As for sufficiency, it is remains to notice that any function~$\phi:\UH\times[0,\infty)\to\Complex$ of the form~\eqref{eq:BRFP_infty} with $q_t$,~$d_t$, and $\pi_t$ satisfying integrability conditions (a) and~(b) is a Herglotz vector field. Indeed, if (a) and~(b) hold, then $\big(a_t,0,q_t,\pi_t\big)$, where $a_t:=d_t-\int_{(0,1)}x\,\pi_t(\di x)$, is a L\'evy family and hence the desired conclusion follows by appealing again to Remark~\ref{RM_HF-DW-infty} and Proposition~\ref{prop:int}.

Finally, \eqref{EQ_der-at-infty} can be obtained from the corresponding formula in \cite[Theorem~D]{GHP} with the help of the Cayley transformation $H(z):=(1+z)/(1-z)$, which maps $\UD$ conformally onto~$\UH$ and takes the boundary point~$1$ to~$\infty$.
\end{proof}

Similarly to the extinction time, we can define the \textsl{explosion time}
\begin{equation}\label{EQ_exp-time}
 T_\infty^s \coloneqq \inf\{t \ge s: Z_t = \infty\}.
\end{equation}
Again we assume that the process $(Z_t)$ is a $D[0,\infty)$-valued random variable. According to~\eqref{eq:derivative_BP}, the probability of finite explosion time is given by
\begin{equation}\label{EQ_explosion-probab}
 \mathbb P^{(s,x)}[T_\infty^s <\infty]=1-\lim_{t\to\infty}e^{-xv_{s,t}(0)}\quad\text{for any~$s\ge0$ and $x>0$.}
\end{equation}
\begin{theorem}\label{TH_as-explosion}
Let $(v_{s,t})\subset\BF$ be an absolutely continuous reverse evolution family with a common BRFP at~$\infty$ and associated Herglotz vector field~$\phi$. Suppose that
\begin{equation}\label{EQ_as-explosion}
  \int_0^{\infty}\!\!\phi(0,t)\exp\!\big(-{\textstyle\int\limits_0^t \phi'(\infty,s)\,\di s}\big)\,\di t~=~-\infty.
\end{equation}
Then for any $s\ge0$,
$$
 \lim\limits_{t\to\infty} v_{s,t}(0)=\infty
$$
and hence for the branching process~$(Z_t)$ with $D[0,\infty)$-paths, associated with~$(v_{s,t})$, we have $\mathbb P^{(s,x)}[T_\infty^s<\infty]=1$ for any ${x>0}$ and any~${s\ge0}$.
\end{theorem}
\begin{proof}
By Theorem~\ref{TH_HF-DW-infty}, $\phi$ is of the form~\eqref{eq:BRFP_infty}. Hence,
$$
\phi(\theta,t)\le\phi_0(\theta,t):=-q_t+d_t\theta
$$
for a.e.\,$t\ge0$ and all $\theta>0$. It is easy to see that $\phi_0$ is a Herglotz vector field and that the  reverse evolution family~$(v_{s,t}^0)$ is given for all $z\in\UH$ and $(s,t)\in\Delta$ by
$$
  v_{s,t}^0(z)\,=\,\alpha(s,t)\,z\,+\int_s^t\!\alpha(s,r)\,q_r\,\di r,~\,~ \alpha(s,t):=\exp\Big(-\int_s^t d_r\,\di r\Big)=v'_{s,t}(\infty).
$$
Arguing as in the proof of Theorem~\ref{TH_infty-to-zero}, we may conclude that $v_{s,t}(\theta)\ge v_{s,t}^0(\theta)$ for any ${\theta>0}$ and ${(s,t)\in\Delta}$. Since ${q_t=-\phi(0,t)}$
and ${d_t=\phi'(0,\infty)}$ for a.e.\,${t\ge0}$, condition~\eqref{EQ_as-explosion} implies that
$$
v_{s,t}(\theta)~>~\int_s^t\!\alpha(s,r)\,q_r\,\di r~\longrightarrow~+\infty\qquad\text{for any~$\theta>0$}
$$
and hence $v_{s,t}(0)\to+\infty$ as $t\to\infty$. By~\eqref{EQ_explosion-probab} it follows that ${\mathbb P^{(s,x)}[T_\infty^s <\infty]=1}$ for any ${s\ge0}$ and ${x>0}$.
\end{proof}

\addtocontents{toc}{\SkipTocEntry}\subsubsection{The Denjoy\,--\,Wolff point in $(0,\infty)$} Recall that according to the definition, for any ${\tau\in(0,\infty)}$ we have $\BF_\tau=\{v\in\BF:v(\tau)=\tau\}$, see e.\,g. Appendix~\ref{App-Fixed} or \hbox{\cite[Sect.\,2.3]{GHP}}.

Assuming that the Laplace exponents belong to~$\BF_\tau$, we can relate the extinction and explosion probabilities, see~\eqref{EQ_ext-time} and~\eqref{EQ_exp-time}, with the position of the DW-point~$\tau$ as follows.
\begin{theorem}\label{TH_infty-to-tau}
Let $\tau \in (0,\infty)$ and $(v_{s,t})$ be an absolutely continuous reverse evolution family in~$\BF_\tau$ with the associated Herglotz vector field~$\phi$ and the associated branching process $(Z_t)$ with $D[0,\infty)$-paths.

\begin{enumerate}[label=\rm(\roman*)]
\item\label{item:hitting_zero}
Suppose that
\begin{equation}\label{EQ_int-phi-two-primes2}
\int_0^{+\infty}\phi''(\infty,t)\,\di t=+\infty.
\end{equation}
Then for any $s\ge0$, $~\lim\limits_{t\to\infty} v_{s,t}(\infty) =\tau~$ and hence, we have
\begin{flalign*}
 &\qquad\qquad\quad
  \mathbb P^{(s,x)}[T_0^s<\infty]=e^{-x \tau} \quad\qquad\text{ for any~${~x>0}~$ and~{}~any~${~s\ge0}$.}
\end{flalign*}

\item\label{item:hitting_infty}
Suppose that
\begin{equation}\label{EQ_int-phi-at-zero}
\int_0^{+\infty}\phi(0,t)\,\di t= - \infty.
\end{equation}
Then for any $s\ge0$, $~\lim\limits_{t\to\infty} v_{s,t}(0) =\tau~$ and hence, we have
\begin{flalign*}
 &\qquad\qquad\quad
   \mathbb P^{(s,x)}[T_\infty^s <\infty]=1- e^{-x \tau} \qquad\text{ for any~$~{x>0}~$ and~{}~any~${~s\ge0}$.}
\end{flalign*}
\end{enumerate}
\end{theorem}
Conditions~\eqref{EQ_int-phi-two-primes2} and~\eqref{EQ_int-phi-at-zero} in the above theorem are essential. Even in the time-homogeneous case, it is not true in general that the boundary values at~$0$ and~$\infty$ tend to the DW-point\hspace{.075em}\footnote{For one-parameter semigroups, convergence to the Denjoy\,--\,Wolff point holds for all \textit{interior} initial conditions, see e.g. \cite[Theorem~8.3.6]{BCD-Book}. In the time-inhomogeneous case, i.e. for evolution families, the situation is more complicated, see e.g. \cite{SantiTiz}.}. Counterexamples can be easily constructed using representation~\eqref{eq:DW_finite_time_hom} below.
\smallskip

For the proof, we first  deduce  an integral representation for Bernstein generators with the specified interior DW-point. Let $P_*$ stand for the entire function given by
\begin{equation*}
P_*(z):=\big(e^{-z}-1+z\big)/z^2\quad \text{for~}~z\in\C\setminus\{0\},\qquad P_*(0)=\tfrac12.
\end{equation*}
\begin{theorem}\label{cor:DW_finite}
Let $\tau \in (0,\infty)$. A function $\phi:\UH\times[0,\infty)\to\C$ is a Herglotz vector field of an absolutely continuous reverse evolution family contained in~$\BF_\tau$ if and only if for a.e.\ ${t\ge0}$ and all ${\zeta\in\UH}$,
\begin{equation} \label{eq:DW_finite2}
\phi(\zeta,t) = (\zeta - \tau)\left[\,\widehat a_t + b_t \zeta + (\zeta-\tau)\int_0^\infty\!\! x^2 e^{-\tau x}P_*\big((\zeta-\tau)x\big) \,\pi_t(\di x)\, \right]
\end{equation}
with some $\widehat a_t,b_t\ge0$ and some L\'evy measures $\pi_t$ on $(0,\infty)$ such that:
\begin{enumerate}[label=\rm(\alph*)]
\item  the functions $t\mapsto\widehat a_t$ and $t\mapsto b_t$ are of class $\Lloc$;
\item $(\pi_t)_{t\ge0}$ is a Borel kernel;
\item
$
  \widehat a_t \ge\,  \tau\int_0^\infty x^2 e^{-\tau x}P_*(-\tau x)\,\pi_t(\di x)~
$
for a.e. $\,t\ge0$.
\end{enumerate}
\end{theorem}
\begin{proof} We observe that the reverse evolution family~$(v_{s,t})$ associated to a Herglotz vector field~$\phi$ is contained in~$\BF_\tau$ with ${\tau\in(0,\infty)}$ if and only if ${(v_{s,t})\subset\BF}$ and ${\phi(\tau,t)=0}$ for a.e.\ ${t\ge0}$. Note also that ${P_*(-\tau x)\ge\tfrac12}$ for all~${x\ge0}$. Bearing these simple facts in mind and using Proposition~\ref{prop:int} with ${\theta_0:=\tau}$ in condition~\ref{item:HVF2}, it suffices to prove that the set of infinitesimal generators $\Gen(\BF_\tau)={\big\{\phi \in \Gen(\BF)\colon\phi(\tau)=0 \big\}}$ coincides with the set of all functions ${\phi:\UD\to\Complex}$ representable as
\begin{equation} \label{eq:DW_finite_time_hom}
\phi(\zeta) = (\zeta - \tau)\left[\,\widehat a + b \zeta + (\zeta-\tau)\int_0^\infty\!\! x^2 e^{-\tau x}P_*\big((\zeta-\tau)x\big) \,\pi(\di x)\, \right]
\end{equation}
with $\widehat a,b\ge0$ and a L\'evy measure~$\pi$ on $(0,\infty)$ such that
\begin{equation}\label{eq:DW_finite10}
\widehat a \ge\,  \tau\int_0^\infty x^2 e^{-\tau x}P_*(-\tau x)\,\pi(\di x).
\end{equation}
To see this, we first suppose that $\phi\in\Gen(\BF_\tau)$. Writing Silverstein's representation~\eqref{gen} for $\phi(\zeta)$ and~$\phi(\tau)$, and taking into account that ${\phi(\tau)=0}$, it is not difficult to obtain \eqref{eq:DW_finite_time_hom} after straightforward calculations of $\phi(\zeta)=\phi(\zeta)-\phi(\tau)$, where
$$
\widehat a  \coloneqq a + b\tau + \int_0^\infty\!\! x \big( \mathbf1_{(0,1)}(x)- e^{-\tau x} \big) \,\pi(\di x).
$$
The fact $\phi(0)=-q \le 0$ implies \eqref{eq:DW_finite10}.

The above reasonings can be traced backwards to prove the converse statement. Alternatively, one can use the characterization of Bernstein generators given in~\cite[Theorem~3]{GHP}.
\end{proof}
\begin{remark}\label{RM_version-of-BP}
Comparing \eqref{eq:DW_finite_time_hom} with the Berkson\,--\,Porta representation for infinitesimal generators in~$\UH$, see \cite[eq.\,(C.2a)]{GHP}, one can rewrite this formula as ${\phi(\zeta)=(\zeta^2-\tau^2)P(\zeta)}$ for all~${\zeta\in\UH}$, where $P$ is a holomorphic function in~$\UH$ given by
\begin{equation*}
P(\zeta):=\frac{\widehat a+b\zeta}{\zeta+\tau}+\frac{\zeta-\tau}{\zeta+\tau}\int_0^{\infty}\!\! x^2 e^{-\tau x}P_*\big((\zeta-\tau)x\big) \,\pi(\di x),\quad \zeta\in\UH.
\end{equation*}
Inequality~\eqref{eq:DW_finite10} is equivalent to ${\Re P\ge0}$ in~$\UH$.  Moreover, combined with the monotonicity of ${\Real\ni\xi\mapsto\xi P_*(\xi)}$, \eqref{eq:DW_finite10} implies that ${P(\theta)\ge (\tau P(0)+b\theta)/(\theta+\tau)}$ for all~${\theta>0}$.
\end{remark}
\begin{proof}[\proofof{Theorem~\ref{TH_infty-to-tau}}]
We employ the same idea as in the proof of Theorem~\ref{TH_infty-to-zero}.
Consider the ``comparison'' Herglotz vector field of the form $\phi_\tau(\zeta,t):=\alpha_t(\zeta^2-\tau^2)$, where $\alpha_t\in L^1_{\mathrm{loc}}\big([0,\infty),[0,\infty)\big)$. The reverse evolution family~$(v^\tau_{s,t})$ associated to~$\phi_\tau$ is given for all~${\zeta\in\UH}$ and all $(s,t)\in\Delta$ by
$$
 v^\tau_{s,t}(\zeta)=\tau+\frac{ 2\tau z(\zeta,\tau)}{e^{2\tau A(s,t)}-z(\zeta,\tau)},\quad
 z(\zeta,\tau):=\frac{\zeta-\tau}{\zeta+\tau}\in\UD,~\,~A(s,t):=\int_s^t\alpha_r\,\di r.
$$
If
$A_\infty:=\int_0^\infty\alpha_r\,\di r=+\infty$,
 then for any $s\ge0$, $v^\tau_{s,t}\to\tau$ \textit{uniformly} in~$\UH$ as ${t\to+\infty}$.\vspace{-.64ex}

\StepP{\ref{item:hitting_zero}} Using representation \eqref{eq:DW_finite2}, we see that $\phi''(\infty,t) = 2b_t \ge0$; hence, the integral in~\eqref{EQ_int-phi-two-primes2} is well-defined.
Choose $\alpha_t:=\phi''(\infty,t)/4=b_t/2$. Then $A_\infty=+\infty$. Moreover, thanks to Remark~\ref{RM_version-of-BP}, we have ${\phi(\theta, t) \ge \phi_\tau(\theta,t)}$ for all $\theta\in[\tau,\infty)$ and a.e.\,$t\ge0$.

To prove~\ref{item:hitting_zero} it is enough to show that
$
\tau \le v_{s,t}(\theta)\le v^\tau_{s,t}(\theta)
$
for all $\theta\ge \tau$ and all $(s,t)\in\Delta$.
The left inequality holds because $ v_{s,t}$ is non-decreasing on ${(0,\infty)}$ and fixes~$\tau$. Lemma~\ref{lem:comparison} implies the other inequality.
Thus, ${v_{s,t}(\infty)\to\tau}$ as ${t\to+\infty}$.
In view of formula~\eqref{EQ_finite-ext-time-probab}, it follows that ${\mathbb P^{(s,x)}[T_0^s<\infty]= e^{-x\tau}}$.\vspace{-.64ex}

\StepP{\ref{item:hitting_infty}} Now we set $\alpha_t:=-\phi(0,t)/(2\tau^2)$. Again ${A_\infty=+\infty}$. Arguing as above, one can show that
${v^\tau_{s,t}(\theta) \le v_{s,t}(\theta)\le \tau}$ for all ${\theta\in(0,\tau]}$ and all ${(s,t)\in\Delta}$.
It follows that $v_{s,t}(0)\to \tau$ as ${t\to+\infty}$. Thus, according to~\eqref{EQ_explosion-probab},
$
\mathbb P^{(s,x)}[T_\infty^s <\infty]=1- e^{-x\tau}.
$
\end{proof}

\section{Branching processes on discrete space}\label{sec3}
Several important results in Section~\ref{sec2} have analogs for time-inhomogeneous branching processes on $\N_0^*\coloneqq \N_0 \cup\{\infty\}$, where $\N_0=\{0,1,2,\dots\}$. We endow $\N_0^*$ with the natural topology to make it a compact space. Note then that $\{\infty\}$ is not an open subset, while for all other (finite) non-negative integers~$n$, the singletons $\{n\}$ are open subsets.
\subsection{Characterization of branching processes}
\,\,To begin with, a \textsl{branching process on~$\N_0^*\,$} is a Markov family with transition kernels\,$(\ell_{s,t})_{(s,t)\in\Delta}$ on~$\N_0^*$ satisfying the following conditions:
\begin{enumerate}[label=\rm(L\arabic*)]
\item\label{L1} $\ell_{s,s}(n,\cdot)= \delta_n(\cdot)$ for every $s \ge 0$ and $n \in \N_0^*$,
\item \label{CKL} $\ell_{s,t}\starM \ell_{t,u}=\ell_{s,u}$ for every $0 \le s \le t \le u$,

\item\label{L3} for every $n \in \N_0$ the map $\Delta \ni (s,t)\mapsto \ell_{s,t}(n,\cdot) \in \Pr(\N_0^*)$ is weakly continuous,
\item \label{eq:branchingL} $\ell_{s,t} (m, \cdot) \ast \ell_{s,t}(n, \cdot) = \ell_{s,t}(m+n,\cdot), \quad  (s,t)\in\Delta,~ m,n \in \N_0,  $   \quad (Branching property)
\item\label{L5} $\ell_{s,t}(0,\cdot)= \delta_0(\cdot)$ for every $(s,t)\in\Delta$,\nopagebreak
\item\label{L6} $\ell_{s,t}(\infty, \cdot) = \delta_\infty(\cdot)$ for every $(s,t)\in\Delta$.
\end{enumerate}
\begin{remark}\label{RM_equiv-def-DSbrancProc}
Condition \ref{eq:branchingL} can be rewritten as $\ell_{s,t}(n,\cdot) = \ell_{s,t}(1,\cdot)^{\ast n}$ for ${n\in\N}$. Also, \ref{L3} is equivalent to the continuity of $\Delta \ni (s,t)\mapsto \ell_{s,t}(n,\{m\}) \in [0,1]$ for every ${n,m \in \N_0}$. It is worth mentioning that the map $(s,t)\mapsto\ell_{s,t}(n, \{\infty\})$ is not required to be continuous.
\end{remark}

The following (sub)\,probability generating function plays the role that the Laplace exponent does for the continuous state case,
\begin{equation}\label{eq:laplace_D}
F_{s,t}(z) := \int_{\N_0} z^n\, \ell_{s,t}(1,\di n) = \sum_{n\ge0} p_{s,t}(n)z^n,  \qquad z \in \UD,
\end{equation}
where $p_{s,t}(n) := \ell_{s,t}(1,\{n\})$. Each $F_{s,t}$ either identically equals~$1$ or belongs to the class
\[
\PGF\coloneqq  \bigg\{\,{\textstyle\sum\limits_{n=0}^{\infty} p_n z^n\colon (\forall n \in \N_0)~ p_n\ge0,\, \textstyle\sum\limits_{n=0}^{\infty}p_n \le 1} \bigg\}\setminus\big\{F\equiv 1\big\},
\]
which is clearly a topologically closed subsemigroup of $\Hol(\UD, \UD)$.

\begin{remark}
In the classical Galton\,--\,Watson processes, the explosion is not allowed, i.e.\ $\ell_{s,t}(n,\{\infty\})=0$ for any ${n\in\N_0}$ and any~$(s,t)\in\Delta$. This corresponds to the subsemigroup~$\Pp$ of~$\PGF$ consisting of functions whose sum of the Taylor coefficients equals~$1$. Working in this setting, Goryainov \cite{Goryainov96R,Goryainov96} studied reverse evolution families of probability generating functions in~$\Pp$. An obvious convenience in considering $\PGF$ is that this semigroup is topologically closed in $\Hol(\UD,\UD)$. Note that this is not the case for $\Pp$. In fact, $\PGF$ is the topological closure of $\Pp$ in $\Hol(\UD,\UD)$.
\end{remark}

Similarly to the continuous-state case, there is a one-to-one correspondence between families of transition kernels of branching processes on~$\N_0^*$ and topological reverse evolution families contained in~$\PGF$.

\begin{theorem}\label{thm:REF_BP_D} Given a family $(\ell_{s,t})_{(s,t)\in\Delta}$ of transition kernels of a branching process on $\N_0^*$, $(F_{s,t})_{(s,t)\in\Delta}$ defined by \eqref{eq:laplace_D} forms a topological reverse evolution family contained in $\PGF$.
Conversely, given a topological reverse evolution family $(F_{s,t})_{(s,t)\in\Delta}$ contained in $\PGF$,  there exists a unique family of transition kernels $(\ell_{s,t})_{(s,t)\in\Delta}$ of a branching process on $\N_0^*$ such that~\eqref{eq:laplace_D} holds.
\end{theorem}
\begin{proof}
Suppose we are given a family $(\ell_{s,t})_{(s,t)\in\Delta}$ of transition kernels of a branching process on~$\N_0^*$. Then it is straightforward to check that the functions $F_{s,t}$ regarded as self-maps of~$\UD$ satisfy the conditions \ref{REF1} and~\ref{REF2} in Definition~\ref{DF_REF}.

Since the Taylor coefficients $p_{s,t}(n)$ of $F_{s,t}$  satisfy ${\sum_{n\ge0}|p_{s,t}(n)|\le1}$,  we have $$\Big|F_{s,t}(z)-\sum_{n=0}^{N-1}p_{s,t}(n)z^n\Big|\le |z|^N\quad\text{for all~$\,z\in\UD~$ and\, all~$\,N\in\Natural$}.$$
Combined with the fact that by~\ref{L3}, $(s,t)\mapsto p_{s,t}(n)$ is continuous for each ${n\in\N_0}$,
this implies the continuity of the map $\Delta\ni(s,t)\mapsto F_{s,t}\in\Hol(\UD,\C)$.

Using now the continuity of $F_{s,t}$ w.r.t. the parameters and following the argument of \textsc{Step~2} in the proof of Theorem~\ref{thm:REF_BP}, we see that $F_{s,t}\subset\Hol(\UD,\UD)$ for any ${(s,t)\in\Delta}$. Thus, $(F_{s,t})$ is a topological reverse evolution family contained in~$\PGF$.

Now suppose we are given a topological reverse evolution family ${(F_{s,t})\subset\PGF}$.
For each ${(s,t)\in\Delta}$, the following relations define a unique transition kernel $\ell_{s,t}$ on $\N_0^*$:
\begin{align*}
&\ell_{s,t}(0,\cdot):=\delta_{0}(\cdot), & &\ell_{s,t}(1,\{m\}):=F_{s,t}^{(m)}(0)/m!\quad\text{for any~$m\in\N_0$,}\\
&\ell_{s,t}(\infty,\cdot):=\delta_{\infty}(\cdot), & &\ell_{s,t}(n,\cdot):=\ell(1,\cdot)^{*n} \quad\text{for any~$n\in\N$}.
\end{align*}
Since $(s,t)\mapsto F_{s,t}\in\Hol(\UD,\UD)$ is continuous, the maps $(s,t)\mapsto F_{s,t}^{(m)}(0)$, ${m\in\N_0}$, are continuous as well. It follows that $(\ell_{s,t})$ satisfies the condition~\ref{L3}, see Remark~\ref{RM_equiv-def-DSbrancProc}.
Checking that the rest of the conditions \ref{L1}\,--\,\ref{L6} hold is straightforward and therefore omitted.
\end{proof}

\subsection{An ODE governing probability generating functions}
According to  \cite[Theorem~30]{Goryainov-survey},  the set $\Gen(\PGF)$ of infinitesimal generators of one-parameter semigroups in $\PGF$ consists of functions
\begin{equation}\label{EQ_PFG-gen}
\Phi(z)\, =\, c \bigg[z-  b_0 -\sum\nolimits_{n\ge2} b_n z^n \bigg], \qquad z \in \UD,
\end{equation}
where $c \ge0$ and $b_0, b_n \ge0$ are such that $b_0 + \sum_{n\ge2}b_n \le 1$. This can be reparametrized into
\begin{equation}\label{EQ_inf-gen-our-param}
\Phi(z)\, =\, q z\, + \sum_{n\in \N_0\setminus\{1\}} \alpha(n) (z-z^n),
\end{equation}
where $q, \alpha(n) \ge 0$ for all $n\in \N_0\setminus\{1\}$ with $\sum_{n\ge2} \alpha(n) <\infty$.  The pair $(q, \alpha)$, where  $\alpha\coloneqq\{\alpha(n)\}_{n\in \N_0\setminus\{1\}}$, is called the \textsl{generating pair} of~$\Phi$.

We say that a family of generating pairs $(q_t,\alpha_t)_{t\ge0}$ is a \textsl{generating family} if the function ${t\mapsto \alpha_t(n)}$ is measurable for each $n\in \N_0\setminus\{1\}$ and if the functions $t\mapsto q_t$ and $t\mapsto \sum_{n\in \N_0\setminus\{1\}}\alpha_t(n)$ are~in~$\Lloc$.

\begin{theorem}\label{thm:main_BP_D}
Let $(F_{s,t})_{(s,t)\in\Delta}$ be an absolutely continuous reverse evolution family contained in $\PGF$.
Then there exists a unique generating family $(q_t,\alpha_t)_{t\ge0}$ such that for every $t \ge0$ and $z \in \UD$ the map $s \mapsto F(s)\coloneqq F_{s,t}(z)$ is the unique solution to the initial value problem
\begin{equation}\label{eq:ODE_D}
\frac{\di }{\di s} F(s)= \Phi(F(s),s)  \quad \text{a.e.~}s \in [0,t]; \qquad F(t)=z,
\end{equation}
where $\Phi\colon \UD\times [0,\infty)\to\C$ is defined by
\begin{equation}\label{eq:BF_generator_D}
\Phi(z,s) =q_s z + \sum_{n\in \N_0\setminus\{1\}} \alpha_s(n)(z-z^n), \qquad z \in \UD, ~ s\ge0,
\end{equation}
i.e.\ the Herglotz vector field associated with $(F_{s,t})_{(s,t)\in\Delta}$ is of the form \eqref{eq:BF_generator_D}.

Conversely, given a generating family $(q_t,\alpha_t)_{t\ge0}$,  the function $\Phi$ defined by \eqref{eq:BF_generator_D} is a Herglotz vector field, and for every $z \in \UD$ and $t\ge0$, the initial value problem \eqref{eq:ODE_D} has a unique solution $s\mapsto F(s)=F(s;t,z)$.  Moreover, the functions $F_{s,t}(z)\coloneqq F(s;t,z)$ form an absolutely continuous reverse evolution family $(F_{s,t})_{(s,t)\in\Delta}\subset \PGF$.
\end{theorem}
\begin{proof} The proof is similar to that of Theorem~\ref{thm:main_BP}; one may use Theorem~\ref{TH_mainBCM1} and \cite[Theorem~2]{GHP}  for the topologically closed semigroup $\PGF\subset\Hol(\UD,\UD)$. One only needs to replace Proposition~\ref{prop:int} with the following argument.

If $\Phi(z,t)$ defined by \eqref{eq:BF_generator_D} is a Herglotz vector field, i.e.\ if $\Phi(z,t)$ is measurable in~$t$ for each $z\in\UD$ and if $t\mapsto\sup_{z\in K}|\Phi(z,t)|$ is locally integrable for every compact $K \subset \UD$, then so does $t\mapsto \Phi^{(n)}(0,t)$ for each ${n\in\N_0}$ thanks to Cauchy's integral formula. Therefore, $t\mapsto\alpha_t(n)=-\Phi^{(n)}(0,t)/n!$ is measurable for each ${n\in\N_0\setminus\{1\}}$, and moreover, since $\Phi'(0,t)= q_t + \sum_{n\ge0, n\ne1} \alpha_t(n)$, the functions $t\mapsto q_t$ and $t\mapsto \sum_{n\in \N_0\setminus\{1\}}\alpha_t(n)$ are locally integrable. Thus, $(q_t,\alpha_t)$ is a generating family.

Conversely, if $(q_t,\alpha_t)_{t\ge0}$ is a generating family, then a straightforward argument implies that  $\Phi(z,t)$ defined by \eqref{eq:BF_generator_D} is measurable in~$t$ for each $z\in\UD$ and that
$t\mapsto \sup_{z\in \UD}|\Phi(z,t)|$ is locally integrable. Thus, $\Phi$ is a Herglotz vector field.
\end{proof}

\subsection{Fixed points, explosion and extinction}
Our next task is to characterize absolutely continuous reverse evolution families with a common boundary regular fixed point and with the common Denjoy\,--\,Wolff point\footnote{For the precise definitions of a boundary fixed point and of the Denjoy\,--\,Wolff point, see e.g.\ Appendix~\ref{App-Fixed}\\ or~\cite[Sect.\,2.3]{GHP}}, boundary or interior. Similarly to the continuous-state case, presence of a common fixed point turns out to be closely related to explosion and extinction probabilities.

As usual, we adopt the following notation: for $n \in \N_0$, $F\in \PGF$ and $\Phi \in \Gen(\PGF)$, let
\[
F^{(n)}(1)\coloneqq  \lim_{\R\ni \xx \to 1-0} F^{(n)}(\xx) \in [0,\infty] \quad \text{and} \quad \Phi^{(n)}(1)\coloneqq  \lim_{\R\ni \xx \to 1-0} \Phi^{(n)}(\xx) \in [-\infty,\infty).
\] The above limits exist because of monotonicity. Note that $F(1)\in [0,1]$ and $\Phi(1) \in [0,\infty)$.  We denote by $\PGF_\sigma$ the set of maps in $\PGF$ having the Denjoy\,--\,Wolff point at~$\sigma$.

Let $(N_t)$ be a branching process on $\N_0^\ast$ having an absolutely continuous reverse evolution family $(F_{s,t})$ and the associated family of transition kernels $(\ell_{s,t})$.
Let
\[
S_r:= \inf\{t\ge0: N_t=r\} \qquad \text{for} \qquad r \in \{0,\infty\}.
\]
The Branching property implies the basic relation
\[
\sum\nolimits_{k\ge0} \ell_{s,t}(n,\{k\})z^k = F_{s,t}(z)^n,\quad z\in\UD,
\]
which further implies
\begin{align*}
\qquad\mathbb P^{(s,n)} [N_t = 0] &= \ell_{s,t}(n,\{0\}) = F_{s,t}(0)^n\\
\text{~\,~and}\qquad\qquad\qquad\qquad\qquad\mathbb P^{(s,n)} [N_t = \infty] &= \ell_{s,t}(n,\{\infty\}) = 1- F_{s,t}(1)^n.
\qquad\qquad\qquad\qquad~
\end{align*}
Because of the monotonicity of the sets $\{N_t =r\}$ w.r.t.\ $t\ge0$, it holds that $$\{S_r <\infty\}~=~ \bigcup\nolimits_{t>0} \{N_t =r\}~=~\bigcup\nolimits_{p \in \N} \{N_p =r\}$$ and hence, for all $n\in \N$ and $s\ge0$, we have
\begin{equation*}
\mathbb P^{(s,n)} [S_0 < \infty] = \lim_{t\to\infty} F_{s,t}(0)^n \quad \text{and} \quad \mathbb P^{(s,n)} [S_\infty < \infty] = \lim_{t\to\infty} 1- F_{s,t}(1)^n.
\end{equation*}

\addtocontents{toc}{\SkipTocEntry}\subsubsection{Finite first moments and boundary regular fixed point at $1$}
Here we establish a characterization of branching processes on $\N_0^*$ with finite first moments analogous to Theorem~\ref{thm:first_moment}. Note that having finite first moments is a sufficient condition for no explosion: $\mathbb P^{(s,n)} [S_\infty < \infty]=0$ for any $n \in \N$ and $s\ge0$.

\begin{theorem}\label{thm:first_moment_D}
Let $(N_t)_{t\ge0}$ be a branching process whose probability generating functions $(F_{s,t})_{(s,t)\in \Delta}$ are characterized by a Herglotz vector field $\Phi$ as in Theorem~\ref{thm:main_BP_D}.
The following conditions are equivalent:
\begin{enumerate}[label=\rm(\arabic*)]

\item\label{first_moment0_D} $\mathbb E^{(s,n)}[N_t]<\infty$ for all $n \in \N$ and $(s,t)\in \Delta$.

\item\label{first_moment1_D} $\mathbb E^{(0,1)}[N_t]<\infty$ for all $t \ge 0$.

\item \label{first_moment1.5_D} $F_{0,t}(1)=1$ and $F'_{0,t}(1)<\infty$, i.e.\ $1$ is a BRFP of $F_{0,t}$ for all $t\ge 0$.

\item \label{first_moment2_D} $F_{s,t}(1)=1$ and $F'_{s,t}(1)<\infty$, i.e.\ $1$ is a BRFP of $F_{s,t}$ for all $(s,t)\in\Delta$.

\item\label{first_moment3_D} $\Phi(1,t)=0$ for a.e.\ $t\ge0$, and $t\mapsto \Phi'(1,t)$ is in $\Lloc$.

\item \label{first_moment4_D} $q_t=0$ for a.e.\ $t\ge0$ and $t\mapsto \sum_{n\in \N_0\setminus\{1\}} (n+1)\alpha_t(n)$ is in $\Lloc$.

\end{enumerate}
If one and hence all of the above conditions hold, then
\[
\mathbb E^{(s,n)}[N_t] = n F_{s,t}'(1)= n\exp \left(- \int_s^t \Phi'(1,r) \,\di r \right)
\]
 for all $n \in \N$ and $(s,t) \in \Delta$.
\end{theorem}
\begin{proof}
The proof is very similar to that of Theorem~\ref{thm:first_moment} and hence it is omitted. We would only mention that instead of formulas~\eqref{eq:derivative_BP} one has to use the relations $F_{s,t}(1) = \mathbb P^{(s,1)} [N_t <\infty]$ and $F_{s,t}'(1)= \mathbb E^{(s,1)}[N_t]$.
\end{proof}

\addtocontents{toc}{\SkipTocEntry}\subsubsection{The Denjoy\,--\,Wolff point at $1$}
Reverse evolution families in $\PGF$ with the common DW-point at~$1$ are characterized here.
The homogeneous case was treated in \cite[Theorem~32]{Goryainov-survey}. We select an equivalent but different form.

\begin{theorem}[\cite{Goryainov-survey}]\label{thm:DW_1_discrete} The set $\Gen(\PGF_1)$ consists of the functions of the form
\begin{equation} \label{eq:DW_1_discrete}
\Phi(z) = \sum_{n\ge0,\, n\ne1}\alpha(n) (z- z^n), \qquad z \in \UD,
\end{equation}
where $\alpha(n)\ge0$ for all~$n\ge2$ and
$
 \alpha(0)\ge \sum_{n\ge2} (n-1) \alpha(n).
$
\end{theorem}
The above theorem follows easily from the representation~\eqref{EQ_inf-gen-our-param}, if we recall that according to \cite[Theorem~3.2]{CD_RACSAM}, ${\sigma=1}$ is the DW-point of a one-parameter semigroup generated by~$\Phi$ if and only if ${\Phi(1)=0}$ and ${\Phi'(1)\ge0}$.

Note also that in the non-autonomous setting, a Herglotz vector field $\Phi$  generates a (reverse) evolution family with a common DW-point at~$\sigma$ if and only if for a.e. ${t\ge0}$, the DW-point of the one-parameter semigroup generated by $\Phi(\cdot,t)$ coincides with~$\sigma$; see e.g.~\cite{BCM1}. Therefore,  Theorems~\ref{thm:main_BP_D} and~\ref{thm:DW_1_discrete}  immediately lead to the following corollary.

\begin{corollary}\label{cor:boundaryDW-D} A function $\Phi:\UD\times[0,\infty)\to\C$ is a Herglotz vector field with associated reverse evolution family contained in $\PGF_1$  if and only if for all $z \in \UD$ and a.e.\,$t\ge0$,
\begin{equation}
\Phi(z,t) =  \sum_{n\ge0, n\ne1}\alpha_t(n) (z- z^n), \qquad z \in \UD,
\end{equation}
where $\alpha_t(n) \ge 0$ for each $t\ge0$ and $n \in \N_0\setminus\{1\}$ such that $t\mapsto \alpha_t(n)$ are all measurable,  $t\mapsto \alpha_t(0)$ is locally integrable and
\begin{equation}
 \alpha_t(0)\ge \sum_{n\ge2} (n-1) \alpha_t(n), \qquad t\ge0.
\end{equation}
\end{corollary}

\addtocontents{toc}{\SkipTocEntry}\subsubsection{Fixed points in $[0,1)$}
Recall that a fixed point  inside the domain of a holomorphic self-map, different from the identity map, is automatically its DW-point. For this reason, a fixed point in $\UD$ of a map in $\PGF$ belongs to its invariant subset $[0,1)$.

The following fact is proved in \cite[Theorem~31]{Goryainov-survey}. We select an equivalent but different form.
For $n\in\Natural$ and $z,w\in\C$, denote
$$
  P^*_n(z,w):=\sum_{k=0}^{n-1}z^kw^{n-k-1}=
  \begin{cases}
    (z^n-w^n)/(z-w) &\text{if~$z\neq w$,}\\
    {}~nw^{n-1} &            \text{if~$z=w$}.
  \end{cases}
$$

\begin{theorem}[\cite{Goryainov-survey}]\label{thm:DW_finite_discrete} Suppose that $\sigma \in [0,1)$. The set $\Gen(\PGF_\sigma)$ consists of the functions of the form
\begin{equation} \label{eq:DW_finite_discrete2}
\Phi(z) = (z - \sigma) \Big[\beta- \sum_{n\ge2}\alpha(n) P^*_n(z,\sigma) \Big] , \qquad z \in \UD,
\end{equation}
where $\alpha(n)\ge0$ for all~${n\ge2}$ and $\beta\ge\sum_{n\ge2} P^*_n(1,\sigma)\alpha(n)$.
\end{theorem}
\begin{remark}
With some computations, one can see that the generating pair $(q,\alpha)$ of $\Phi$ in the above theorem is given by
\begin{equation*}
q = \Phi(1) = (1-\sigma) \Big[\beta - \sum_{n\ge2} P^*_n(1,\sigma)\alpha(n) \Big] \qquad \text{and} \qquad \alpha(0) = \sigma \beta - \sum_{n\ge2} \alpha(n)\sigma^n.
\end{equation*}
The parameters $\alpha(n)$ for $n\ge2$ in~\eqref{EQ_inf-gen-our-param} and in~\eqref{eq:DW_finite_discrete2} are the same.
\end{remark}

Theorems~\ref{thm:main_BP_D} and~\ref{thm:DW_finite_discrete} imply the following
\begin{corollary}\label{cor:DW_finite_discrete}
Let $\sigma\in[0,1)$. A function $\Phi:\UD\times[0,\infty)\to\C$ is a Herglotz vector field with associated reverse evolution family contained in $\PGF_\sigma$  if and only if for all $z \in \UD$ and a.e.\,$t\ge0$,
\begin{equation}  \label{eq:DW_finite_discrete}
\Phi(z,t) = (z - \sigma) \Big[\beta_t - \sum_{n\ge2}\alpha_t(n) P^*_n(z,\sigma) \Big],
\end{equation}
where  $t\mapsto\beta_t$ and $t\mapsto\alpha_t(n)$, $n\ge2$, are non-negative real-valued functions such that:
\begin{itemize}
\item[\rm (a)] $t\mapsto \alpha_t(n)$, $n\ge2$, are measurable;\vskip.3ex
\item[\rm (b)] $t\mapsto\beta_t$ is $\Lloc$;\vskip.3ex
\item[\rm (c)] $\sum\nolimits_{n\ge2} P_n^*(1,\sigma)\alpha_t(n)\,\le\,\beta_t\,\,$ for all $t\ge0$.
\end{itemize}
\end{corollary}

\noindent We conclude the subsection with an analog of Theorem~\ref{TH_infty-to-tau} in the discrete state space setting.

\begin{theorem}\label{TH_infty-to-sigma}
Let $\sigma \in [0,1)$ and $(F_{s,t})$ be an absolutely continuous reverse evolution family in~$\PGF_\sigma$ with the associated Herglotz vector field~$\Phi$ and the associated branching process $(N_t)$ on~$\N_0^\ast$.
\begin{enumerate}[label=\rm(\roman*)]
\item\label{item:hitting_zero_discrete}
Suppose that
\begin{equation}\label{EQ-Phi}
\int_0^{+\infty}\Phi'(\sigma,t)\,\di t=+\infty.
\end{equation}
Then for any $s\ge0$,
\begin{equation}\label{EQ_conv-to-sigma}
\lim\limits_{t\to\infty} F_{s,t}(0) =\sigma
\end{equation}
 and hence we have $\mathbb P^{(s,n)}[S_0<\infty]=\sigma^n$ for any $n \in \N$ and any~${s\ge0}$. Moreover, if $\sigma\neq0$, then \eqref{EQ_conv-to-sigma} implies~\eqref{EQ-Phi}.\medskip

\item\label{item:hitting_infty_discrete}
Suppose that
\begin{equation}\label{EQ_int-Phi-1}
\int_0^{+\infty}\Phi(1,t)\,\di t= \infty.
\end{equation}
Then for any $s\ge0$, $$\lim\limits_{t\to\infty} F_{s,t}(1) =\sigma$$ and hence we have $\mathbb P^{(s,n)}[S_\infty<\infty]=1- \sigma^n$ for any $n \in \N$ and any~${s\ge0}$.
\end{enumerate}
\end{theorem}
\begin{proof}
Assertion~\ref{item:hitting_zero_discrete} is immediate if~$\sigma=0$. So we assume that $\sigma\in(0,1)$. Fix an arbitrary~${s\ge0}$.
Notice that every $F_{s,t}$, ${t\ge s}$, is an increasing self-map of~$(-1,1)$ fixing~$\sigma$. By the invariant form of the Schwarz\,--\,Pick Lemma, see e.g. \cite[Theorem~6.4]{BM2007}, we have ${|F_{u,t}(x)-\sigma|\le|x-\sigma|}$ for any ${x\in(-1,1)}$ and any~$(u,t)\in\Delta$. It follows that ${t\mapsto F_{s,t}(x)}$ is monotonic and hence has a limit as $t\to\infty$ pointwise w.r.t. ${x\in(-1,1)}$. Since holomorphic self-maps of~$\UD$ form a normal family and taking into account that each $F_{s,t}$ fixes the same point ${\sigma\in\UD}$, it further follows, see e.g. \cite[\S\,II.7]{Goluzin}, that $F_{s,t}$ converges, as ${t\to\infty}$, locally uniformly in~$\UD$ to some $F_{s,\infty}\in\Hol(\UD,\UD)$. Moreover, by the corollary of Hurwitz's Theorem on univalent functions, see e.g. \cite[p.\,5]{Duren}, the limit function is either univalent in~$\UD$, or $F_{s,\infty}\equiv\sigma$. Since ${F'_{s,t}(\sigma)=\exp\big(-\int_s^t\Phi'(\sigma,r)\di r\big)\to F'_{s,\infty}(\sigma)}$ as ${t\to\infty}$, we see that the limit function is constant and, in particular $F_{s,\infty}(0)=\sigma$, if and only if~\eqref{EQ-Phi} holds. Otherwise, ${F_{s,\infty}(0)\neq F_{s,\infty}(\sigma)=\sigma}$ by univalence.  This proves~\ref{item:hitting_infty_discrete}.

The proof of~\ref{item:hitting_infty_discrete} is similar to that of Theorems~\ref{TH_infty-to-zero} and~\ref{TH_infty-to-tau}.  For ${t\ge0}$ and ${z\in\UD}$, let
\[
 p_t := \frac{\Phi(1,t)}{1-\sigma}\quad\text{and}\quad \Phi_0(z,t):=(z-\sigma)p_t=(z-\sigma)\Big[\beta_t -  \sum_{n\ge2} P_n^*(1,\sigma)\alpha_t(n)\Big].
\]
It follows from conditions (a)\,--\,(c) in Corollary~\ref{cor:DW_finite_discrete}, that ${p_t\ge0}$ for a.e.~${t\ge0}$ and that ${t\mapsto p_t}$ is of class~$\Lloc$. Therefore, again by Corollary~\ref{cor:DW_finite_discrete} with $\beta_t$ and $\alpha_t(n)$'s replaced by $p_t$ and~$0$, respectively, $\Phi_0$ is also a Herglotz vector field in~$\UD$ and the associated reverse evolution family~$(F_{s,t}^0)$ is contained in~$\PGF_\sigma$. In fact, $F^0_{s,t}(z) = \sigma + (z-\sigma) e^{- \int_s^t p_r \,\di r}$ for any ${(s,t)\in\Delta}$ and any $z\in\UD$.

Taking into account that $P_n^*(x,\sigma)\le P_n^*(1,\sigma)$ for any ${x\le1}$ and all~$n\ge2$, representation~\eqref{eq:DW_finite_discrete} implies that $\Phi(x,t)\ge\Phi_0(x,t)$ for a.e.~${t\ge0}$ and all~$x\in[\sigma,1)$. Thus, arguing as in the proof of Theorem~\ref{TH_infty-to-tau}\,(i), we see that ${\sigma\le F_{s,t}(x)\le F_{s,t}^0(x)}$ for any ${(s,t)\in\Delta}$ and any $x\in[\sigma,1)$. Passing to the limit as ${x\to1-0}$ yields the desired conclusion, because
if \eqref{EQ_int-Phi-1} holds, then ${F^0_{s,t}(1)\to\sigma}$ as ${t\to\infty}$.
\end{proof}

\subsection{Spatial embeddability}
The notion of ``spatial embedding'' --- i.e.\ embedding of a given discrete-state branching process into a continuous-state one~--- was discussed in \cite[Section 4]{MJ1958} (where the term ``extension'' was used for spatial embedding). In our setting, this remarkable notion can be defined as follows.

\begin{definition}\label{def:spat-embd}
A branching process on $\N_0^*$ with transition kernels $(\ell_{s,t})_{(s,t)\in\Delta}$ is said to be \textsl{spatially embeddable} into a branching process on $[0,\infty]$ if there exists a branching process on $[0,\infty]$ with transition kernels $(k_{s,t})_{(s,t)\in\Delta}$ such that for all $n \in \N_0^*$, $(s,t)\in\Delta$,
\begin{equation}\label{EQ_embedd-cond}
k_{s,t}(n,B)=\ell_{s,t}(n,B\cap \N_0^*)\quad\text{for any Borel set ${B\subset[0,\infty]}$}.
\end{equation}
\end{definition}

Below we establish
a fairly simple
characterization
of spatial embeddability in terms of the generating functions, which is further equivalent to a sort of monotonicity of the process.

\begin{theorem}\label{thm:non-decreasingD}
Let $(N_t)_{t\ge0}$  be a branching process on~$\N_0^*$ with probability generating functions $(F_{s,t})_{(s,t)\in\Delta}$. The following conditions are equivalent:
\begin{enumerate}[label=\rm(\roman*)]
\item \label{item:E1} $(N_t)$ is spatially embeddable into a branching process on $[0,\infty]$;
\item\label{item:E2} $F_{s,t}(0)=0$ for any $(s,t)\in\Delta$, i.e. $(F_{s,t})\subset\PGF_0$;
\item\label{item:E3} $\ell_{s,t}(n, \{m\})=0$ for any $(s,t)\in\Delta$ and $m,n \in \N_0$ with $m<n$;
\item\label{item:E4} $\mathbb P^{(s,n)}[ N_t \le N_u]=1$ for any $0 \le s \le t \le u$ and $n \in \N_0$.
\end{enumerate}
Moreover, if $(N_t)_{t\ge0}$ has \cadlag paths, then the above conditions are further equivalent to that $(N_t)$ has non-decreasing paths almost surely.
\end{theorem}
\begin{proof}First note that the last statement is an easy consequence of~\ref{item:E4}. For the rest, the proof is divided into five steps.

\Step{1: \rm \ref{item:E1}~$\Rightarrow$~\ref{item:E2}} Suppose that $(N_t)$ is embeddable into a branching process $(Z_t)$ on~${[0,\infty]}$. As before, denote by $(v_{s,t})$ the family of its Laplace exponents defined by~\eqref{eq:laplace}. It is easy to see that $e^{-v_{s,t}(\theta)}=F_{s,t}(e^{-\theta})$ for all ${(s,t)\in\Delta}$ and any ${\theta>0}$. Fix some $(s,t)\in\Delta$. Extending the last equality by holomorphicity, we have
\begin{equation}\label{EQ_covering}
F_{s,t}(e^{-\zeta})=\exp(-v_{s,t}(\zeta))\quad\text{for all~$\zeta\in\UH$.}
\end{equation}
In particular, we see that
\begin{equation}\label{EQ_plus 2_pi_m}
v_{s,t}(\theta+2\pi i)=v_{s,t}(\theta)+2\pi mi\quad\text{ for all ${\theta>0}$}
\end{equation}
and some $m\in\mathbb Z$ not depending on~$\theta$.
Note that $m\neq0$ because $v_{s,t}$ is univalent in~$\UH$. Indeed, by Theorem~\ref{thm:REF_BP}, $(v_{s,t})_{(s,t)\in\Delta}$ is a topological reverse evolution family.
Taking into account the relationship between evolution families and reserve evolution families, see \cite[Remark~2.8]{GHP}, the univalence follows by \cite[Proposition~2.4]{CD_top-LTh}.

We claim that $v_{s,t}(\infty):=\lim_{\theta\to+\infty}v_{s,t}(\theta)=+\infty$. Recall that this limit exists and belongs to~$(0,+\infty]$ because $v_{s,t}\in\BF$ and hence it is a positive non-decreasing function on $(0,\infty)$. By Lindel\"of's Theorem, see e.g.\ \cite[Theorem~1.5.7 on p.\,27]{BCD-Book} or \cite[Theorem~9.3 on p.\,268]{Pombook75}, $v_{s,t}$ has angular limit at~$\infty$. It follows that the limit of the l.h.s.\ of~\eqref{EQ_plus 2_pi_m} exists and equals~$v_{s,t}(\infty)$. This contradicts the fact that~${m\neq0}$ unless $v_{s,t}(\infty)=+\infty$.

Now passing to the limit as $\Real\in\zeta\to+\infty$ in~\eqref{EQ_covering}, we get $F_{s,t}(0)=0$, which proves~\ref{item:E2}.

\Step{2: \rm \ref{item:E2}~$\Rightarrow$~\ref{item:E1}} Suppose that $F_{s,t}(0)=0$ for any~$(s,t)\in\Delta$. Since by Theorem~\ref{thm:REF_BP_D}, $(F_{s,t})$ is a topological evolution family, arguing as above we see that $F_{s,t}$'s are univalent in~$\UD$. It follows that $F_{s,t}(\UD^*)\subset\UD^*:=\UD\setminus\{0\}$. Therefore, the mappings $F_{s,t}$ can be lifted from to $\UD^*$ to~$\UH$ w.r.t. the covering ${\UH\ni\zeta\mapsto e^{-\zeta}\in\UD^*}$. In other words, for any $(s,t)\in\Delta$, there exists $v_{s,t}\in\Hol(\UH,\UH)$ such that~\eqref{EQ_covering} holds. Moreover, since $F\big((0,1)\big)\subset(0,1)$, we may assume that $v_{s,t}\big((0,\infty)\big)\subset(0,\infty)$.

The most difficult part is to show that $(v_{s,t})\subset\BF$. Since $(F_{s,t})$ is a reverse evolution family and since $v_{s,t}(\theta)=-\log F_{s,t}(e^{-\theta})$ for all ${\theta>0}$ and all ${(s,t)\in\Delta}$, with the help of the identity principle for holomorphic functions we may conclude that $(v_{s,t})$ satisfies conditions~\ref{REF1} and~\ref{REF2} in Definition~\ref{DF_REF}. Furthermore, recall that ${F_{s,t}(0)=0}$ and ${F'_{s,t}(0)>0}$ for all~$(s,t)\in\Delta$, with the strict inequality taking place thanks to the univalence of~$F_{s,t}$. Hence, by the Schwarz lemma, from the identity $F_{0,t}={F_{0,s}\circ F_{s,t}}$, $(s,t)\in\Delta$, it follows that $\lambda(t):=F_{0,t}'(0)$ is a positive non-increasing function. Choose\footnote{To construct a concrete example, consider the increasing homeomorphism~$f$ of $[0,\infty)$ defined by $f(s):={s - \lambda(s)+1}$. Since $f(s_2) - f(s_1) \ge s_2 - s_1$ whenever  $0 \le s_1 \le s_2$, the function $u:=f^{-1}$ is Lipschitz continuous and so is $\lambda (u(t))= u(t)-t+1$.} an increasing homeomorphism $u$ of $[0,\infty)$ onto itself such that $\lambda\circ u$ is locally absolutely continuous on~${[0,\infty)}$.  Then according to \cite[Remark~2.8]{GHP} and \cite[Theorem~7.3]{BCM1} the formula $\tilde F_{s,t}:=F_{u(s),u(t)}$ defines an absolutely continuous reverse evolution family~$(\tilde F_{s,t})$. Thanks to the equivalence between \ref{item:E1} and~\ref{item:E2} in~\cite[Proposition~4.3]{CDG_decr}, it follows that the family $(\tilde v_{s,t})$ defined by $\tilde v_{s,t}:=v_{u(s),u(t)}$ for all ${(s,t)\in\Delta}$ is an absolutely continuous reverse evolution family in~$\UH$.

Let $\Phi$ be the Herglotz vector field associated with the reverse evolution family~$(\tilde F_{s,t})$. Applying Theorem~\ref{thm:main_BP_D} to $(\tilde F_{s,t})$ and using relation~\eqref{EQ_covering} we obtain the following ODE for~$(\tilde v_{s,t})$:
\begin{equation}\label{EQ_ODE-for-lifting}
\frac{\di \tilde v_{s,t}(\zeta)}{\di s}=\phi(\tilde v_{s,t}(\zeta),s)\quad\text{~for all $t>0$, all $\zeta\in\UH$, and a.e. $s\in[0,t]$,}
\end{equation}
where $~\phi(\zeta,s) := - e^{\zeta} \Phi(e^{-\zeta}, s)$. Equation~\eqref{EQ_ODE-for-lifting} means that $\phi$ is the Herglotz vector field associated with~$(\tilde v_{s,t})$.

Using representation~\eqref{eq:BF_generator_D} and taking into account that $\Phi(0,s)=0$ because $0$~is a fixed point for~$(F_{s,t})$, we find that
\begin{equation}
\phi(\zeta,s)=-q_s-\sum_{n\ge2}\alpha_s(n)\big(1-e^{-(n-1)\zeta}\big)= -q_s-\int_{(0,\infty)}\big(1-e^{-x\zeta}\big)\,\pi_s(\di x)
\end{equation}
for all~$s\ge0$ and all~$\zeta\in\UH$, where $\pi_s:=\sum_{k\in\Natural}\alpha_s(k+1)\delta_k$. Since  for each~${s\ge0}$ the series $\sum_{n\ge2}\alpha_s(n)$  converges, the measures~$\pi_s$ satisfy the integrability condition $\int_{(0,\infty)}\min\{\lambda,1\}\,\pi_s(\di\lambda)<\infty$. Hence, by \cite[Corollary~3.7]{GHP}, $\phi(\cdot,s)\in\Gen(\BF_\infty)$ for all~${s\ge0}$. The semigroup $\BF_\infty$ is topologically closed in~$\Hol(\UH,\UH)$. Therefore,  ${(\tilde v_{s,t})\subset\BF_\infty}$ by \cite[Theorem~2]{GHP}.

It immediately follows that $(v_{s,t})$ is a topological reverse evolution family contained in~$\BF_\infty$.
Thus, by Theorem~\ref{thm:REF_BP} and Remark~\ref{RM_branching-process-from-trans-probas}, there exists a branching process~$(Z_t)$ on~$[0,\infty]$ whose Laplace exponents are exactly $(v_{s,t})$. Using the relations~\eqref{EQ_covering}, \eqref{eq:laplace_D}, \eqref{eq:laplace} and taking into account that a bounded Borel measure on $[0,\infty)$ is uniquely determined by its Laplace transform, see e.g.~\cite[Proposition~1.2]{SSV12}, we conclude that $k_{s,t}(1,B)=\ell_{s,t}(1,B\cap\N_0^*)$ for any~$(s,t)\in\Delta$ and any Borel set $B\subset[0,\infty]$. Thanks to the branching properties~\ref{eq:branching} and~\ref{eq:branchingL}, the latter immediately implies~\eqref{EQ_embedd-cond}. This proves~\ref{item:E1}.

\Step{3:  \rm \ref{item:E2}~$\Rightarrow$~\ref{item:E3}} Note that \ref{item:E2} is equivalent to $\ell_{s,t}(1,\{0\})=0$ for any $(s,t) \in \Delta$. The branching property $\ell_{s,t}(n,\cdot):=\ell(1,\cdot)^{*n}$ implies
\[
\ell_{s,t}(n,\{m\}) = \sum_{\substack{k_1 + k_2 + \cdots + k_n =m \\ k_i \ge 0}} \ell_{s,t}(1,\{k_1\})\ell_{s,t}(1,\{k_2\}) \cdots \ell_{s,t}(1,\{k_n\}).
\]
In case $m<n$, there exists some $i$ such that $k_i=0$, and hence  $\ell_{s,t}(n,\{m\})=0$.

\Step{4:  \rm \ref{item:E3}~$\Rightarrow$~\ref{item:E4}} Denote $\Delta^*:= \{(k,m): k,\,m\in\N_0^*,~k \le m\}$.
For any $0 \le s \le t \le u$ and $n \in \N_0$, using \eqref{eq:FD} we have
\begin{align*}
\mathbb P^{(s,n)}[ N_t \le N_u]
&=\mathbb P^{(s,n)}[(N_t, N_u)\in\Delta^*] = \int_{\Delta^*}\!\ell_{s,t}(n, \di k)\,\ell_{t,u}(k,\di m)  \label{eq:MK} \\[1.25ex]
&= \sum_{(k,m) \in \Delta^*}\ell_{s,t}(n, \{k\})\,\ell_{t,u}(k,\{m\})
=   \sum_{k \in \N_0^*}\ell_{s,t}(n, \{k\})~=~1.
\end{align*}

\Step{5:  \rm \ref{item:E4}~$\Rightarrow$~\ref{item:E2}}  This implication is immediate. Indeed, since by~\ref{item:E4}, $$1= \mathbb P^{(s,1)}[ N_s \le N_t]  = \ell_{s,t}(1, \N \cup \{\infty\}),$$ we have $F_{s,t}(0)  = \ell_{s,t}(1,\{0\})= 1- \ell_{s,t}(1, \N \cup \{\infty\})=0$ as desired.
\end{proof}

\newpage
%%%%%%%%%%%%%%%%%%%%%%%%%%%555
\begin{appendix}
\section{Proof of Theorem~\ref{thm:cadlag}}\label{App-Th2.7}
Theorem~\ref{thm:cadlag}, stated in Section~\ref{SS_topological}, seems to be a rather standard result. However, for the purpose of completeness, we provide here a detailed proof.

\begin{proof}[\proofof{Theorem~\ref{thm:cadlag}}]  The uniqueness is standard. Fix an $s\ge0$.
For each cylinder set $$A:=(\Xc_{t_1})^{-1}(B_1) \cap \cdots \cap (\Xc_{t_n})^{-1}(B_n),$$ where $s \le t_1 \le \cdots\le t_n, B_ i \in \mathcal B ([0,\infty])$, the value of $\mathbb Q^{(s,x)}(A)$ is determined only by $(k_{s,t})$. Indeed, in view of~\eqref{eq:FD}, we have
\begin{equation*}
\mathbb Q^{(s,x)}(A)=\int_{B_1 \times \cdots \times B_n} k_{s,t_1}(x, \di x_1)\, k_{t_1,t_2}(x_1, \di x_2) \cdots k_{t_{n-1}, t_n}(x_{n-1}, \di x_n).
\end{equation*}
It only remains to recall that the cylinder sets form a $\pi$-system and that this $\pi$-system generates the whole $\sigma$-field $\mathcal G_{[s,\infty)}$. Therefore, every probability measure on $\mathcal G_{[s,\infty)}$ is uniquely determined by its evaluations on the cylinder sets, see \cite[Chapter~II, Corollary~4.7 on p.\,93]{RW}.\medskip

For the existence, let $\big((Z_t)_{t\ge0},\Omega, \mathcal F, (\mathcal F_I)_I, (\mathbb P^{(s,x)})_{s\ge0,\, x \in [0,\infty]}\big)$ be as in Remark~\ref{RM_branching-process-from-trans-probas}. In order to define measures $\mathbb Q^{(s,x)}$ we first construct a \cadlag modification $(Y_{s,t})_{t\ge s}$ of $(Z_t)_{t\ge s}$.
Recall that by Theorem~\ref{thm:REF_BP}, the Laplace exponents $v_{s,t}$ associated to the transition kernels~$k_{s,t}$ via~\eqref{eq:laplace} form a topological reverse evolution family.

\Step{1: Construction of $Y_{s,t}$}
Throughout this step, we fix an arbitrary $s \ge0$ and some ${\theta>0}$.
	For a while we fix also some ${N \in \N}$ with ${N> s}$. Consider the stochastic process $$V_t^{(N)}:=\exp(-v_{t,N}(\theta)Z_t),\quad s \le t < N,$$ where $V_t^{(N)}$ is set to be~$0$ if ${Z_t=\infty}$.  The process $\big(V_t^{(N)}\big)_{t \in [s,N)}$  is a martingale w.r.t. $\big((\mathcal F_{[s,t]})_{t\in [s,N)}, \mathbb P^{(s,x)}\big)$ for any $x \in [0,\infty]$, because thanks to the Markov property~\eqref{eq:Markov} and~\ref{REF2},  for any ${s \le t \le u < N}$ we have
	\[
	\mathbb E^{(s,x)}\big[\exp(
-{v_{u,N}}(\theta) Z_u) \big| \mathcal F_{[s,t]}\big]  = \exp\big(- Z_t v_{t,u} (v_{u,N}(\theta) )\big)
                                                       = \exp(- v_{t,N}(\theta) Z_t)
	\]
$\mathbb P^{(s,x)}\text{\,-\,a.s.}$ By \cite[Chapter~II, Proposition~67.6 on p.\,173]{RW}, the stochastic process $(V^{(N)}_{t})_{t\in [s,N)}$ is a martingale with respect to the usual augmentation $(\widehat{\mathcal F}_{[s,t]}^x)_{t\ge s}$ of the filtration $(\mathcal F_{[s,t]})_{t\ge s}$ w.r.t.\ $t$, i.e.\
		the $\sigma$-field generated by $\bigcap_{u > t} \mathcal F_{[s,u]}$ and all the $\mathbb P^{(s,x)}$-null sets of $\mathcal F_{[s,\infty)}$. Denote by $\widehat{\mathbb P}^{(s,x)}$ the unique extension of $\mathbb P^{(s,x)}$ to $\widehat{\mathcal F}_{[s,\infty)}^x$. By the Regularization Theorem (see e.g.~\cite[p.~170--173]{RW} or \cite[Theorem~1-(2) on p.\,67]{DM82}), the process $(V^{(N)}_{t})_{t\in [s,N)}$ admits a \cadlag modification  $(W^{(N)}_{s,t})_{t\in [s,N)}$ which is a martingale w.r.t. the right-continuous filtration~${(\widehat{\mathcal F}_{[s,t]}^x)_{t\ge s}}$.\smallskip

For $p\in \{0,1\}$ we let
$
T^{(N)}_{s,p}(\omega) := \big\{t\in[s,N): W^{(N)}_{s,t-}(\omega) =p \,\text{~or~}\, W^{(N)}_{s,t}(\omega)=p \big\}
$
and consider the event $\Omega^{(N)}_{s,p}$ defined as the union of $\big\{\omega \in \Omega\,:\, T^{(N)}_{s,p}(\omega)=\emptyset\}$ and
\[
 \Big\{\omega \in \Omega\,:\, T^{(N)}_{s,p}(\omega)\neq\emptyset~\text{\,and\,}~W^{(N)}_{s,t}(\omega) = p \text{~for all~} t \in [\inf T^{(N)}_{s,p}(\omega), N)\Big\}.
\]
Note that $W^{(N)}_{s,t}$ is $\mathcal F_{[s,\infty)}$-measurable for all $t\in [s,N)$. Therefore, by assertion~(i) of \cite[Chapter~II, Theorem~78.1 on p.\,191]{RW}, $\Omega^{(N)}_{s,p} \in \mathcal F_{[s,\infty)}$. Moreover, in view of assertion~(ii) of the same theorem, the fact that $(W^{(N)}_{s,t})_{t\in [s,N)}$ is a martingale (w.r.t. the filtered probability space indicated above) allows us to conclude  that $\mathbb P^{(s,x)}[\Omega^{(N)}_{s,p}]=\widehat{\mathbb P}^{(s,x)}[\Omega^{(N)}_{s,p}]=1$,  ${p\in\{0,1\}}$.\smallskip

For $t \in [s,N)$, we set
\[
Y_{s,t}^{(N)}(\omega) :=
\begin{cases}
\displaystyle -\frac1{v_{t,N}(\theta)}\log W^{(N)}_{s,t}(\omega),& \omega \in  \Omega^{(N)}_{s,0} \cap \Omega^{(N)}_{s,1}, \\
\quad 1, & \omega \in \Omega \setminus (\Omega^{(N)}_{s,0} \cap \Omega^{(N)}_{s,1}),
\end{cases}
\]
where $\log$ is naturally extended to the homeomorphism from $[0,1]$ onto $[-\infty,0]$. The process $(Y_{s,t}^{(N)})_{t\in[s,N)}$ is a modification of~$(Z_t)_{t\in [s,N)}$ since $W_{s,t}^{(N)}$ is a modification of $V_t^{(N)}$.

For each pair of integers $M,N$ with $s<M<N$ there exists an event $\Omega_s^{(M,N)} \in \mathcal{F}_{[s,\infty)}$ of probability~1 such that $Y_{s,t}^{(M)}(\omega)=Y_{s,t}^{(N)}(\omega)$ for all $s\le t <M$ and all ${\omega \in \Omega_s^{(M,N)}}$. Indeed,
\[
1\,=\,\mathbb P^{(s,x)}\big[\forall t \in [s,M)\cap \mathbb Q,\, Y_{s,t}^{(M)} = Y_{s,t}^{(N)}\big]\,=\,\mathbb P^{(s,x)}\big[\forall t \in [s,M),\, Y_{s,t}^{(M)} = Y_{s,t}^{(N)}\big],
\]
where the first equality follows by the fact that $Y_{s,t}^{(M)}$ and $Y_{s,t}^{(N)}$ are both modifications of $(Z_t)$ and the second equality follows by the \cadlag property.
Let $\Omega_s:=\bigcap_{M, N\in \N, s<M<N}\big(\Omega^{(N)}_{s,0} \cap \Omega^{(N)}_{s,1}\cap\Omega_s^{(M,N)}\big) \in \mathcal{F}_{[s,\infty)}$. For each $t \in [s,+\infty)$ we select $N\in\N, N>t$ and set
\[
Y_{s,t}(\omega) : =
\begin{cases} Y_{s,t}^{(N)}(\omega), & \omega \in \Omega_s, \\
1, & \omega \in \Omega \setminus \Omega_s.
\end{cases}
\]
This definition does not depend on the choice of~$N$, and $(Y_{s,t})_{t\ge s}$ is a modification of~$(Z_t)_{t\ge s}$. It remains to notice that  by construction, once a sample path $(Y_{s,t}(\omega))_{t\ge s}$ hits the boundary point $0$ or $\infty$, it stays there afterwards. Thus every sample path of $(Y_{s,t})_{t\ge s}$ belongs to the space $D[s,\infty)$.

\Step{2: Construction of $\mathbb Q^{(s,x)}$} Fix for a while some~$s\ge0$ and $x\in[0,\infty]$.
Note that $Y_{s,t}$ is ${\big(\mathcal F_{[s,\infty)},\mathcal B([0,\infty])\big)}$-measurable. It follows that the map $\Phi_s\colon \Omega \to D[0,\infty)$ defined by
\[
(\Phi_s(\omega))(t) := \begin{cases}  Y_{s,t}(\omega), & t \ge s, \\
\,1, & 0 \le t <s,
\end{cases}
\]
 is $\big(\mathcal F_{[s,\infty)}, \mathcal G_{[s,\infty)}\big)$-measurable. We set $\mathbb Q^{(s,x)}$ to be the push-forward of the measure~$\mathbb P^{(s,x)}$ w.r.t. $\Phi_s$.  Then for every ${n\in \N}$, ${s \le t_1 \le \cdots \le t_n}$ and ${B_1,\dots, B_n \in \mathcal B([0,\infty])}$, we have
\begin{align*}
 \mathbb Q^{(s,x)}[\Xc_{t_1} \in B_1, \dots, \Xc_{t_n} \in B_n]
 &= \mathbb P^{(s,x)}[Y_{s,t_1} \in B_1, \dots, Y_{s,t_n} \in B_n]  \\
 &=  \mathbb P^{(s,x)}[Z_{t_1} \in B_1, \dots, Z_{t_n} \in B_n],
\end{align*}
where the second equality holds because $(Y_{s,t})$ is a modification of $(Z_t)$.

Regard the coordinate process $(\Xc_t)$ as equipped with the measures $\mathbb Q^{(s,x)}$. Then the above equality trivially implies  \eqref{eq:initial} and, by standard arguments, also implies the Markov property \eqref{eq:Markov} for $(\Xc_t)$, see e.g.\ \cite[\S8.3.1, p.\ 137]{Wen81}. Therefore,  $\big((\Xc_t)_{t\ge0}, D[0,\infty), \mathcal G, (\mathcal G_I)_I, (\mathbb Q^{(s,x)})_{s\ge0, x \in [0,\infty]}\big)$ is a branching process with transition kernels~$(k_{s,t})$, as desired.
\end{proof}

\section{Fixed points of holomorphic self-maps}\label{App-Fixed}
Here we gather a few definitions and the most basic results concerning iteration of holomorphic self-maps. The proofs can be found, e.g., in the monographs \cite[Chapter~1.4]{Abate:book}, \cite[Chapter~5]{Abate2} and \cite{BCD-Book}, where this topic is covered in depth.

Fix a holomorphic self-map $v\colon\UD\to\UD$ different from $\id_\UD$. According to the classical Denjoy\,--\,Wolff Theorem (see e.g. \cite[\S1.8]{BCD-Book}),
\begin{itemize}
\item[(i)] either $v$ is an elliptic automorphism of~$\UD$, i.e.\ a M\"{o}bius transformation mapping $\UD$ onto itself and having one fixed point~$\tau$ in~$\UD$,\smallskip
\item[(ii)] or there exists a (unique) point $\tau$ in the closure $\overline\UD$ of~$\UD$ such that
$$v^{\circ n}:=\underbrace{\,v\circ\ldots\circ v}_{n\text{~times}}~\to~ \tau$$\vspace{-2ex}

\noindent locally uniformly in~$\UD$ as ${n\to+\infty}$.
\end{itemize}
In both cases, $\tau$ is referred to as the \textsl{Denjoy\,--\,Wolff point} of~$v$, or in short, the \textsl{DW-point}.
If the DW-point $\tau$ of $v\in\Hol(\UD,\UD)\setminus\{\id_\UD\}$ lies inside~$\UD$, then it is the unique fixed point of~$v$ in~$\UD$ and ${|v'(\tau)|\le1}$ with the equality in case~(i) and strict inequality in case~(ii).
If~${\tau\in\partial\UD}$, then $v$ has no fixed points in~$\UD$. In this case, $\tau$ can be still considered a fixed point in the following sense.  A \textsl{boundary fixed point} of~$v$ is a point ${\sigma\in\partial\UD}$ such that ${\anglim_{z\to\sigma}v(z)=\sigma}$, where $\anglim$ denotes the so-called angular or non-tangential limit, see e.g. \cite[Section~1.5]{BCD-Book}. For any boundary fixed point~$\sigma$, the limit
$$
v'(\sigma):=\anglim_{z\to\sigma}\frac{v(z)-\sigma}{z-\sigma},
$$
known as the \textsl{angular derivative} of $v$ at~$\sigma$,
exists finite or infinite. If $v'(\sigma)\neq\infty$, then $\sigma$ is said to be a \textsl{boundary regular fixed point} of~$v$, or \textsl{BRFP} for short.

If the DW-point $\tau\in\partial\UD$, then $\tau$ is a BRFP and $v'(\tau)\in(0,1]$. Conversely, if $\sigma$ is BRFP but not the DW-point of~$v$, then $v'(\sigma)\in(1,\infty)$.

All mentioned above can be extended to holomorphic self-maps of~$\UH$ with the help of the conformal mapping
$$
\UD\ni\zeta\mapsto\frac{1+\zeta}{1-\zeta}\in\UH.
$$
The only significant difference is that for the boundary point $\sigma=\infty$ the angular derivative is defined by
$$
v'(\infty)=\anglim_{z\to\infty}\frac{v(z)}{z}.
$$
By the Wolff Lemma (also known as the half-plane version of the Julia Lemma, see e.g. \cite[Theorem~B]{GHP}), the above limit exists and it is a non-negative real number. The point~$\infty$ is a BRFP of~$v\in\Hol(\UH,\UH)\setminus\{\id_\UH\}$ if and only if~$v'(\infty)\neq0$. Similarly, $\infty$ is the DW-point if and only if~${v'(\infty)\ge1}$.

\begin{remark}\label{RM_the-same-FPs}
If $(v_t)$ is a one-parameter semigroup in $\Hol(\UD,\UD)$ or in $\Hol(\UH,\UH)$, then all the elements of~$(v_t)$ different from the identity map have the same DW-point and the same boundary (regular and irregular) fixed points, see e.g. \cite[pp.\,218--220, 327--328]{BCD-Book}.
\end{remark}
\begin{remark}\label{RM_DW-Bernstein}
The most interesting case for the purposes of this paper is when ${v:\UH\to\UH}$ is a Bernstein function. Since in such a case $(0,\infty)$ is mapped into itself, the DW-point of $v$ belongs to $[0,\infty]$.
Notice also that existence of the angular limit trivially implies the existence of the corresponding radial limit. On the other hand, by Lindel\"of's Theorem, see e.g.\ \cite[Theorem~1.5.7 on p.\,27]{BCD-Book}, if $v\in\Hol(\UH,\UH)$ has a radial limit at some boundary point, then it has an angular limit at the same point. It follows that $\sigma$ is a boundary fixed point of~$v\in\BF$ if and only if ${v(\sigma)=\sigma}$ and that in such a case the angular derivative of~$v$ at~$\sigma$ coincides with $v'(\sigma)$, where $v(\sigma)$ and $v'(\sigma)$ are to be understood as limits along the semiaxis~$(0,\infty)$; see~\eqref{EQ_BF-limits}.
\end{remark}

\section{A comparison Lemma}
In several proofs we have used a technical lemma, which should be known to specialists in differential equations. Being unable to find a suitable reference to the existing literature, we state it below together with a proof.
\begin{lemma} \label{lem:comparison}
Let $T\in(0,\infty]$, $I\subset\Real$. Let $f\colon [0,T)\to I$ and $g\colon [0,T)\to I$ be two locally absolutely continuous functions with ${f(0)=g(0)}$. Suppose that there exists a function ${h\colon I \times [0,T)\to\Real}$ such that:
\begin{itemize}
\item[\rm (i)] for any $x\in I$, the function $h(x,\cdot)$ is measurable on~$[0,T)$;\smallskip
\item[\rm (ii)] there exists a locally integrable function $p\colon [0,T)\to [0,\infty)$ such that $$\qquad\big|h(x,t)-h(y,t)\big| \le p(t)|x-y|$$ for a.e. $t \in [0,T)$ and all $x,\,y \in I$;\smallskip
\item[\rm (iii)] $g'(t)=h(g(t),t)$ and $f'(t)\le h(f(t),t)$ for a.e.\ $t\in[0,T)$.
\end{itemize}\smallskip
Then $f \le g$ on $[0,T)$.
\end{lemma}
\begin{proof}
Suppose to the contrary that $f(t_\#) > g(t_\#)$ for some $t_\# \in [0,T)$. Set $$t_\ast := {\max \{t \in [0, t_\#): f(t) = g(t)\}},$$ which  exists because  $f$ and~$g$ are continuous and ${f(0)= g(0)}$. Then ${f(t)>g(t)}$ for all $t \in (t_\ast, t_\#]$. Replacing~$t_\#$ with a point in~$(t_*,T)$ close enough to~$t_*$, we may assume that ${\int_{t_\ast}^{t_\#}\!p(s)\di s<1}$.
Taking into account that by a standard argument, see e.g. \cite[Chapter~VIII, \S8]{SansoneV2},  $s\mapsto h(f(s),s)$ is integrable on~$[0,t_\#]$, for every $t \in(t_\ast, t_\#]$, we have
\begin{align*}
f(t) - g(t)
&= \int_{t_\ast}^t [f'(s) - g'(s)]\, \di s  \\
&\le \int_{t_\ast}^t [h(f(s),s) - h(g(s),s)]\, \di s \\
&\le  \int_{t_\ast}^t p(s) [f(s) - g(s)]\, \di s \\
&\le \sup_{t \in [t_\ast, t_\#]}[f(t)- g(t)] \int_{t_\ast}^t p(s) \, \di s.
\end{align*}
It follows that
\[
\sup_{t \in [t_\ast, t_\#]}[f(t)- g(t)]  \le \sup_{t \in [t_\ast, t_\#]}[f(t)- g(t)] \int_{t_\ast}^{t_\#} \!p(s) \, \di s\,  <\sup_{t \in [t_\ast, t_\#]}[f(t)- g(t)],
\]
which is a contradiction.
\end{proof}

\end{appendix}

%%%%%%%%%%%%%%%%%%%%%%%%%%%555

%\addtocontents{toc}{\SkipTocEntry}
\section*{Acknowledgments}
The authors are deeply indebted to the anonymous referees for their very careful reading and valuable comments and remarks, which helped to improved the text and simplify some of the proofs.

Part of this work was finished during T.H.'s stays in
Guanajuato thanks to the hospitality of CIMAT and Octavio Arizmendi.

%\addtocontents{toc}{\SkipTocEntry}
\section*{Funding}
This work is supported by the Research Institute for Mathematical Sciences, an International Joint
Usage/Research Center located in Kyoto University, and by JSPS Grant-in-Aid for Young Scientists
19K14546, JSPS Scientific Research 18H01115 and JSPS Open Partnership Joint Research Projects grant
no. JPJSBP120209921.

 P.G. is partially supported by GNSAGA INdAM (\textsl{Istituto Nazionale di Alta Matematica ``Francesco Severi'',} Italy).

\end{document}